\documentclass[a4paper, 11pt]{article}
\usepackage{amssymb}
\usepackage{bm}
\usepackage{textcomp}
\usepackage{mathrsfs}
\usepackage{amsmath}
\usepackage{amsfonts}
\usepackage[T1]{fontenc}
\usepackage[utf8]{inputenc}    
\usepackage[T1]{fontenc}       
\usepackage{amsthm}
\usepackage{graphicx}
\usepackage{amsmath}

\usepackage{amsthm}
\usepackage{array}
\usepackage{graphicx}
\usepackage{cases}
\usepackage{enumerate}
\usepackage{clrscode}
\usepackage{listings}
\usepackage[ruled,vlined,english,linesnumbered]{algorithm2e}
\usepackage{url}
\usepackage{enumitem}
\usepackage{pgffor}
\usepackage{float}

\usepackage{xcolor}
\usepackage{soul} 

\usepackage{enumerate,enumitem,todonotes}
\makeatletter
\def\th@plain{%
  \itshape 
}
\makeatother

\makeatletter
\renewenvironment{proof}[1][\proofname]{\par
  \pushQED{\qed}%
  \normalfont \topsep6\p@\@plus6\p@\relax
  \trivlist
  \item[\hskip\labelsep
        \bfseries
    #1\@addpunct{.}]\ignorespaces
}{%
  \popQED\endtrivlist\@endpefalse
}
\makeatother

\newtheorem{theorem}{Theorem}[section]

\numberwithin{equation}{section}

\newtheorem{thm}{Theorem}[section]

\newtheorem{lemma}[thm]{Lemma}
\newtheorem{example}[thm]{Example}

\newtheorem{rem}[thm]{Remark}
\newtheorem{op}[thm]{Open Question}

\numberwithin{equation}{section}

\usepackage[varg]{txfonts}
\usepackage{graphicx, epsfig, subfigure}
\usepackage{enumerate} 
\usepackage[square, numbers, sort&compress]{natbib} 
\ifx\pdfoutput\undefined
 \usepackage[dvipdfm,%
  pdfstartview=FitH, 
  bookmarks=true,%
  bookmarksnumbered=true, 
  bookmarksopen=true, 
  plainpages=false,%
  pdfpagelabels,%
  colorlinks=true, 
  linkcolor=blue, 
  citecolor=blue,%
  urlcolor=black,
  pdfborder=001]{hyperref}
  \AtBeginDvi{}  
\else
 \usepackage[pdftex,%
  pdfstartview=FitH, 
  bookmarks=true,%
  bookmarksnumbered=true, 
  bookmarksopen=true, 
  plainpages=false,%
  pdfpagelabels,%
  colorlinks=true, 
  linkcolor=blue, 
  citecolor=blue,%
  urlcolor=black,
  pdfborder=001]{hyperref}
\fi

\numberwithin{equation}{section}

\title{\LARGE Between proper and square colorings of planar graphs with maximum degree at most four}

\author{
Xujun Liu\thanks{Department of Applied Mathematics, School of Mathematics and Physics, Xi'an Jiaotong-Liverpool University, Suzhou, Jiangsu Province, 215123, China, \texttt{xujun.liu@xjtlu.edu.cn}; the research of X. Liu was supported by the National Natural Science Foundation of China under grant No.~12401466.}
\and
Zihui Xu\thanks{School of Mathematics and Physics, Xi'an Jiaotong-Liverpool University, Suzhou, Jiangsu Province, 215123, China, \texttt{zihui.xu2503@student.xjtlu.edu.cn}.}
\and
Xin Zhang\thanks{School of Mathematics and Statistics, Xidian University, Xi'an, Shaanxi Province, 710126, China, \texttt{xzhang@xidian.edu.cn}.}
}

\date{}

\begin{document}

\maketitle

\begin{abstract}
An $i$-independent set is a vertex set whose pairwise distance is at least $i+1$. A proper (square) $k$-coloring of a graph $G$ is a partition of its vertex set into $k$ independent ($2$-independent) sets. A packing $(1^{j}, 2^k)$-coloring of a graph $G$ is a partition of $V(G)$ into $j$ independent sets and $k$ $2$-independent sets. It can be viewed as intermediate colorings between proper and square coloring. Wegner conjectured in 1977 that every planar graph with maximum degree at most four is square $9$-colorable. Bousquet, Deschamps, de Meyer, and Pierron proved an upper bound of $12$, which is the current best result toward the conjecture of Wegner. In this paper, we prove two analogue results that every planar graph with maximum degree at most four is packing $(1,2^{10})$-colorable and packing $(1^2,2^7)$-colorable. 
\end{abstract}

\section{Introduction}
An $i$-independent set is a vertex set whose pairwise distance is at least $i+1$. For a non-decreasing sequence $S = (s_1, \ldots, s_k)$ of positive integers, a packing $S$-coloring of a graph $G$ is a partition of $V(G)$ into $V_1, \ldots, V_k$ such that each $V_i$ is $s_i$-independent. This notion was introduced by Goddard and Xu~\cite{GX1}, and later studied by many researchers (e.g., see~\cite{BKL1, BF1, BKRW1, CL1, FKL1, GT1, LW1, LZZ1, MT1, MT2, ZM1}). We use $(s_1^{t_1}, \ldots, s_k^{t_k})$ to denote the tuple where $s_i$ appears with multiplicity $t_i$ for every $i$. For example, $(1,2,2,2,3,3)$ is denoted as $(1,2^3,3^2)$. A proper $k$-coloring is equivalent to a packing $(1^k)$-coloring, while a square $k$-coloring is equivalent to a packing $(2^k)$-coloring. The chromatic number, $\chi(G)$, of a graph $G$ is the minimum $k$ such that $G$ has a proper $k$-coloring. The minimum $k$ such that a graph $G$ has a square $k$-coloring is denoted by $\chi_2(G)$.

The famous Four-Color Theorem states that every planar graph is packing $(1^4)$-colorable. Wegner~\cite{W1} conjectured in 1977 that if $G$ is a planar graph of maximum degree at most $\Delta$, then
\[
\chi_2(G) \le 
\begin{cases} 
7 & \mbox{ if $\Delta = 3$}\\
\Delta + 5 & \mbox{ if $ \Delta\in\{4,5,6, 7\}$}\\
\left\lfloor \frac{3}{2} \Delta \right\rfloor + 1 & \mbox{ if $\Delta \ge 8$}. \\
\end{cases}
\]

Thomassen~\cite{T1}, and independently Hartke, Jahanbekam, and Thomas~\cite{HJT1} proved Wegner's conjecture for $\Delta = 3$. The conjecture remains open for all $\Delta \ge 4$. Bousquet, Deschamps, de Meyer, and Pierron~\cite{BDMP1} proved the current best upper bound of $12$ when $\Delta = 4$. van den Heuvel and McGuinness~\cite{HM1} showed that every planar graph $G$ with maximum degree $\Delta$ has $\chi_2(G) \le 2 \Delta + 25$. Amini, Esperet, and van den Heuvel~\cite{AEH1}, and independently Havet, van den Heuvel, McDiarmid, and Reed~\cite{HHMR1} showed that the conjecture holds asymptotically, namely, $\chi_2(G) \le \frac{3}{2} \Delta + o (\Delta)$ as $\Delta \to \infty$. The list version of square $k$-coloring is also well-studied (e.g., see~\cite{BLP1, CK1}). In particular, the result of Cranston and Kim~\cite{CK1} implies that every connected subcubic graph except the Petersen graph has a square $8$-coloring, and this bound is tight.

We focus on packing $(1^j,2^k)$-colorings, which can be viewed as intermediate colorings between proper coloring and square coloring. Extensive research exists for packing $(1^j, 2^k)$-coloring of subcubic graphs. Gastineau and Togni~\cite{GT1} proved that every subcubic graph has a packing $(1,2^6)$-coloring and a packing $(1^2,2^3)$-coloring. Their results are tight since the Petersen graph has no packing $(1,2^5)$-coloring and no packing $(1^2,2^2)$-coloring. Liu and Wang~\cite{LW1} proved that every subcubic planar graph is packing $(1,2^5)$-colorable. Their result is the best possible due to the existence of subcubic planar graphs that are not packing $(1,2^4)$-colorable. Very recently, Mortada and Togni~\cite{MT2} showed that every graph with maximum degree at most $\Delta$ is packing $(1^{\Delta-1},2^{\Delta})$-colorable. In contrast to the celebrated Four-Color Theorem, Choi and Liu~\cite{CL1} proved that for every fixed $k$, there is a planar graph that has no packing $(1^3,2^k)$-coloring. Gr\"otzsch's Theorem states that the minimum $g$, for which every planar graph with girth at least $g$ is packing $(1^3)$-colorable, is $4$. Choi and Liu~\cite{CL1} showed that the minimum $g$, for which every planar graph with girth at least $g$ has a packing $(1^2, 2^k)$-coloring for some finite $k$, is exactly $7$.  

In this paper, we initiate the study of packing $(1^j, 2^k)$-coloring for planar graphs with maximum degree at most four. We study the two cases in which $j$ is fixed to be $1$ or $2$. As a corollary of the result of Bousquet, Deschamps, de Meyer, and Pierron~\cite{BDMP1}, every planar graph with maximum degree at most four has a packing $(1,2^{11})$-coloring and a packing $(1^2, 2^{10})$-coloring. We first prove in Section 2 that every planar graph with maximum degree at most four is packing $(1,2^{10})$-colorable.

\begin{theorem}\label{maintheorem1}
Every planar graph with maximum degree at most four is packing $(1,2^{10})$-colorable. 
\end{theorem}

Let $k_i$ be defined as the minimum integer $k_i'$ such that every planar graph with maximum degree at most four is packing $(1^i,2^{k_i'})$-colorable, where $i \in \{1,2\}$. Theorem~\ref{maintheorem1} proves $k_1 \le 10$. We show $k_1 \ge 6$.

\begin{example}\label{example1}
The graph $G_1$ in Figure~\ref{sharp} is a planar graph with maximum degree four. Its independence number and diameter are both equal to two. Therefore, a $1$-color can be used at most twice, and a $2$-color can be used at most once. It follows that $G_1$ has no packing $(1, 2^5)$-coloring and no packing $(1^2, 2^3)$-coloring.
\end{example}

\begin{figure}
\vspace{-5mm}
\begin{center}
 \includegraphics[scale=0.7]{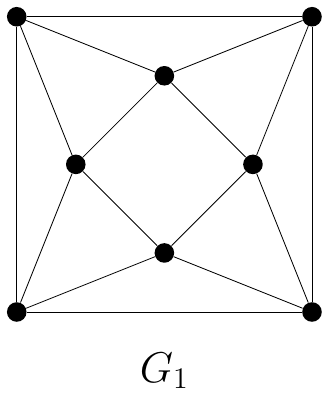}
 \vspace{-3mm}
\caption{An example for the sharpness of Theorems 1.1 and 1.2.}\label{sharp}
\end{center}
\vspace{-8mm}
\end{figure}

As a corollary of Theorem~\ref{maintheorem1}, every planar graph with maximum degree at most four is packing $(1^2,2^9)$-colorable. We further show that $k_2 \le 7$.

\begin{theorem}\label{maintheorem2}
Every planar graph with maximum degree at most four is packing $(1^2,2^7)$-colorable.    
\end{theorem}

Example~\ref{example1} also shows that $k_2 \ge 4$. We prove Theorem~\ref{maintheorem1} in Section 2 and Theorem~\ref{maintheorem2} in Section 3. In this paper, we say a vertex $v$ sees a $1$-color if $v$ has a neighbour using that color. Similarly, we say $v$ sees a $2$-color if there is a vertex $u$ using that $2$-color and $u$ is within distance two from $v$. A $2$-conflict refers to the situation where two vertices sharing the same $2$-color and are at distance at most $2$ in $G$.

\begin{figure}[ht]
\vspace{-30mm}
 \hspace{-30mm} \includegraphics[scale=1.05]{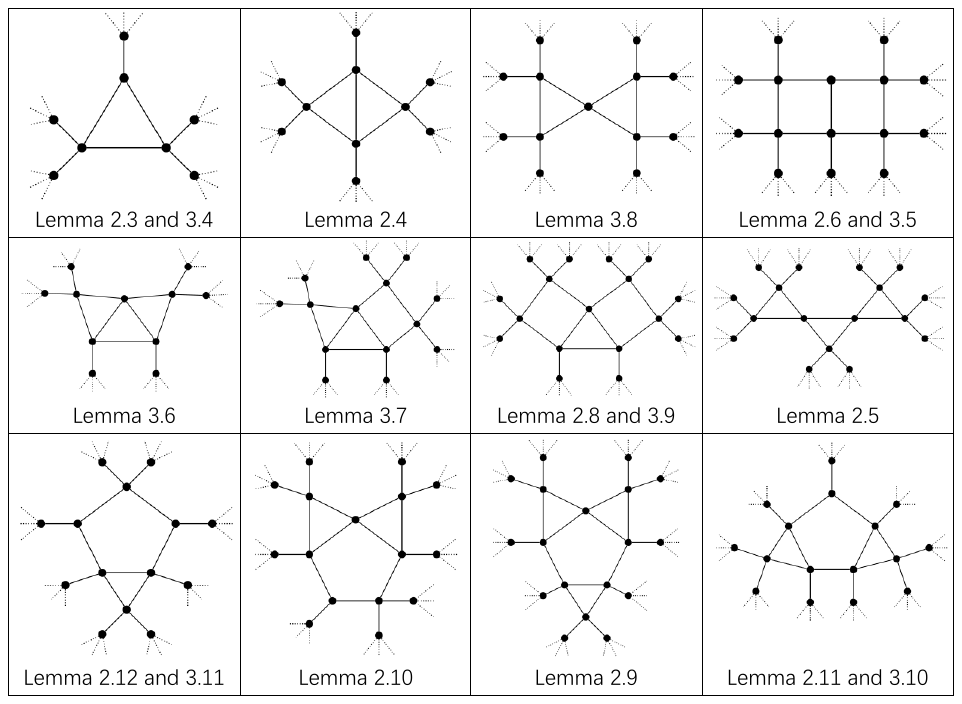}
  \vspace{-165mm}
\caption{Reducible Configurations for Theorems 1.1 and 1.2.}\label{configurations}
\vspace{-2mm}
\end{figure}

\section{Packing $(1,2^{10})$-coloring of planar graphs of maximum degree at most $4$}
In this section, we show that every planar graph with maximum degree at most four has a packing $(1,2^{10})$-coloring. We call a packing $(1,2^{10})$-coloring {\em good coloring} in this section. We use $1$ to denote the $1$-color and let $\{2_1,2_2,2_3,2_4,2_5,2_6,2_7,2_8,2_9,2_{10}\}$ be the set of $2$-colors.  

\textbf{Proof of Theorem~\ref{maintheorem1}:} Suppose not, i.e., there exist planar graphs with maximum degree at most four that have no packing $(1,2^{10})$-coloring. We take such a graph $G$ with minimum $|V(G)|+|E(G)|$.

\begin{lemma}\label{mindegree}
$\delta(G) \ge 3$.    
\end{lemma}

\begin{proof}
We first show there is no $1$-vertex. Suppose not, i.e., $u$ is a $1$-vertex in $G$ with its unique neighbour $u_1$. We delete $u$ from $G$ to obtain a planar graph $G'$ with $\Delta(G') \le 4$. By the minimality of $G$, $G'$ has a good coloring $f$. Since there are at most $4$ vertices within distance two from $u$ and we have in total $11$ colors, we can extend $f$ to $G$. Then we show $\delta(G) \ge 3$. Suppose not, i.e., $u$ is a $2$-vertex with neighbours $u_1,u_2$. We delete $u$ from $G$ and add the edge $u_1u_2$ (if it is not already in $G$) to obtain a planar graph $G'$ with $\Delta(G') \le 4$. By the minimality of $G$, $G'$ has a good coloring $f$. Since there are at most $8$ vertices within distance two from $u$ and we have in total $11$ colors, we can extend $f$ to $G$. Note that $f$ has no $2$-conflict. 
\end{proof}


\begin{lemma}\label{no33}
Two $3$-vertices cannot be adjacent.    
\end{lemma}

\begin{proof}
Suppose not, i.e., $u, v$ are two $3$-vertices with $uv \in E(G)$. Let $N(u) = \{v, u_1, u_2\}$ and $N(v) = \{u, v_1, v_2\}$. We identify $u$ and $v$ (call this new vertex $u'$) to obtain a planar graph $G'$ with $\Delta(G') \le 4$. By the minimality of $G$, $G'$ has a good coloring $f$. If $f(u') = 1$, then we can use $1$ at $u$. Note that each of $v_1$ and $v_2$ must see color $1$ in $G'-u'$ at least once, since otherwise we recolor $v_1$ or $v_2$ with $1$. Thus, at most eight $2$-colors appear within distance two from $v$, and we can color $v$ with an available $2$-color. Therefore, we can assume $f(u') \neq 1$, say $f(u') = 2_1$. We may assume $1 \in \{f(u_1), f(u_2), f(v_1), f(v_2)\}$, say $f(u_1) = 1$, since otherwise we can recolor $u'$ with $1$, which is already done. We use $2_1$ at $u$. By a similar argument, $v$ has at least one available $2$-color to use. We can always extend $f$ to $G$. 
\end{proof}

By Lemma~\ref{mindegree} and~\ref{no33}, a $3$-cycle can have at most one $3$-vertex. 
We strengthen this result by showing that, in fact, no $3$-cycle can have a $3$-vertex. 

\begin{lemma}\label{3cycle4vertices}
Every vertex of a $3$-cycle must be a $4$-vertex.    
\end{lemma}

\begin{proof}
Suppose $u_1u_2u_3u_1$ is a $3$-cycle with $u_1$ being a $3$-vertex. Let $N(u_1) = \{v_1, u_2, u_3\}$. We delete $u_1$ and add $v_1u_2$ (if it is not already in $G$) to obtain a planar graph $G'$ with $\Delta(G') \le 4$. By the minimality of $G$, $G'$ has a good coloring $f$. Since there are at most $10$ vertices within distance two from $u_1$ and we have $11$ available colors, we can extend $f$ to $G$. Note that $f$ has no $2$-conflict.
\end{proof}

\begin{lemma}\label{two3cycles}
Two $3$-cycles cannot share an edge.
\end{lemma}
\begin{proof}
Suppose $u_1u_2u_3u_1$ and $u_1u_2u_4u_1$ are two $3$-cycles sharing the edge $u_1u_2$. By Lemma \ref{3cycle4vertices}, $u_1$ and $u_2$ are 4-vertices. Let $N(u_1)=\{u_2,u_3,u_4,v_1\}$ and $N(u_2)=\{u_1,u_3,u_4,v_2\}$. We identify $u_1$ and $u_2$ (call this new vertex $u'$) to obtain a planar graph $G'$ with $\Delta(G') \le 4$. By the minimality of $G$, $G'$ has a good coloring $f$. We extend $f$ to $G$. If $f(u')$ is a $2$-color, say $f(u') = 2_1$, then we can use $2_1$ at $u_2$. Note that each of $v_1$, $u_3$, and $u_4$ must see color $1$ in $G'-u'$ at least once, since otherwise we can recolor $v_1$, $u_3$ or $u_4$ with $1$ so that this condition holds. Thus, at most nine $2$-colors appear within distance two from $u_1$, and we can color $u_1$ with an available $2$-color. Therefore, we can assume $f(u') = 1$. We color $u_2$ with $1$. Note that $v_1$ must see color $1$ in $N[v_1]$ at least once, since otherwise we recolor $v_1$ with $1$. Moreover, $u_3,u_4$ can only see color $1$ at $u_2$, since otherwise at most nine $2$-colors appear within distance two from $u_1$, and we can color $u_1$ with an available $2$-color. Then we uncolor $u_2$ and recolor $u_3, u_4$ with $1$, and the number of available $2$-colors at $u_1, u_2$ is at least $2,1$. We can then color $u_2$ and $u_1$ in that order.
\end{proof}

By Lemma~\ref{two3cycles}, two $3$-cycles cannot share an edge. However, they may share a vertex. In Lemma~\ref{superbowtie}, we show that a $3$-cycle can share vertices with at most one $3$-cycle.

\begin{lemma}\label{superbowtie}
A $3$-cycle can share vertices with at most one $3$-cycle.   
\end{lemma}

\begin{proof}
Let $u_1u_2u_3u_1$ and $u_1u_4u_5u_1$ be two $3$-cycles sharing the vertex $u_1$. By Lemma~\ref{3cycle4vertices}, $u_2,u_3,u_4,u_5$ are all $4$-vertices. Let $N(u_2) = \{u_1, u_3, u_6, u_7\}, N(u_3) = \{u_1, u_2, u_8, u_9\}$. In order to prove the lemma, we need to show $u_6u_7 \notin E(G)$. Suppose not, i.e., $u_6u_7 \in E(G)$. By Lemma~\ref{two3cycles}, $u_4,u_5 \notin \{u_6, u_7\}$ and $u_3u_4, u_2u_5 \notin E(G)$. We delete $u_1$ and add $u_2u_5, u_3u_4$ to obtain a planar graph $G'$ with $\Delta(G') \le 4$. By the minimality of $G$, $G'$ has a good coloring $f$. We extend $f$ to $G$.

\textbf{Case 1:} All of $f(u_2), f(u_3), f(u_4), f(u_5)$ are not $1$. We color $u_1$ with $1$, and we are done. 

\textbf{Case 2:} Two of $f(u_2), f(u_3), f(u_4), f(u_5)$ are $1$. We have two subcases up to symmetry.

\textbf{Case 2.1:} $f(u_3) = f(u_5) = 1$. One of $u_6$ and $u_7$ must see color $1$ at least once, since otherwise we can recolor $u_6$ or $u_7$ to $1$ and color $u_1$ with an available $2$-color. We uncolor $u_2$, and $u_1, u_2$ have at least $1,2$ available $2$-colors. We color them in order.

\textbf{Case 2.2:} $f(u_2) = f(u_4) = 1$. If $1 \in \{f(u_8), f(u_9)\}$, then $u_1$ has an available $2$-color to use, and we are done. Otherwise, we uncolor $u_2$, and color $u_3$ with $1$ and $u_2$ with an available $2$-color. By symmetry, we can then apply the same argument as in Case 2.1 to $u_3$ and $u_4$.

\textbf{Case 3:} One of $f(u_2), f(u_3), f(u_4), f(u_5)$ is $1$. If $f(u_2) = 1$ or $f(u_3) = 1$, then, by Case 2, each of $u_4,u_5$ must see color $1$ at least once, and $u_1$ has an available $2$-color. Otherwise, $f(u_4) = 1$ or $f(u_5) = 1$. By Case 2, each of $u_2,u_3$ must see color $1$ at least once, and $u_1$ has an available $2$-color.
\end{proof}

\begin{lemma}\label{3vertex4cycle}
A $3$-vertex can belong to at most one $4$-face.
\end{lemma}
\begin{proof}
Suppose $u_1u_2u_3u_4u_1$ and $u_3u_4u_5u_6u_3$ are two $4$-faces sharing the edge $u_3u_4$ with $u_4$ being a $3$-vertex. We claim $u_5 \neq u_1$ and $u_2 \neq u_6$, since otherwise one of $u_1u_2u_3u_4u_1$ and $u_3u_4u_5u_6u_3$ is a separating $4$-cycle and thus is not a $4$-face. By Lemma~\ref{3cycle4vertices}, $u_1u_5 \notin E(G)$, $u_2 \neq u_5$, and $u_1 \neq u_6$. We delete $u_4$ from $G$ and add the edge $u_1u_5$ to obtain a planar graph $G'$ with $\Delta(G') \le 4$. By the minimality of $G$, $G'$ has a good coloring of $f$. If $f(u_1), f(u_3), f(u_5)$ are all $2$-colors, then we color $u_4$ with $1$. Otherwise, at most nine $2$-colors can appear within distance two from $u_4$. We color $u_4$ with an available $2$-color.
\end{proof}

\begin{rem}\label{saveonecolor}
We use the following color-saving argument very often. Let $G$ be a graph and $f$ be a partial good coloring. Let $u$ be a vertex such that $u$ has been assigned a color by $f$.

At least one vertex in $N[u]$ is colored with $1$ by $f$, since otherwise we recolor $u$ with $1$.

\end{rem}


\begin{lemma}\label{no434}
A $3$-face can share edges with $4$-faces via at most one edge.
\end{lemma}

\begin{proof}
Let $u_1u_2u_3u_4u_1$ be a $4$-face and $u_1u_4u_5u_1$ be a $3$-face such that they share the edge $u_1u_4$. By Lemma~\ref{3cycle4vertices}, $u_1, u_4, u_5$ are all $4$-vertices. Let $N(u_1)=\{u_2, u_4, u_5, u_7\}$, $N(u_4)=\{u_1, u_3, u_5, v_4\}$, and $N(u_5)=\{u_1, u_4, u_6, v_5\}$. By Lemma~\ref{two3cycles}, $v_5,u_6 \notin \{u_2, u_3, u_7, v_4\}$ and $u_7 \notin \{u_3, v_4\}$. We need to prove $u_6u_7\notin E(G)$. Suppose not, i.e., $u_6u_7\in E(G)$. We delete $u_1, u_4, u_5$ and add $u_2u_7, u_3v_4, u_6v_5$ (if they are not already in $G$) to obtain a planar graph $G'$ with $\Delta(G')\le4$. By the minimality of $G$, $G'$ has a good coloring $f$. 


\textbf{Case 1:} At least one of $u_2, u_3, u_6, u_7$ is colored with $1$. By symmetry, we may assume $1 \in \{f(u_2), f(u_3) \}$. By Remark~\ref{saveonecolor}, each of $v_4, v_5, u_6, u_7$ must see color $1$ at least once. Thus, the number of available $2$-colors at $u_1, u_4, u_5$ is at least $2, 1, 2$. We split the proof depending on whether $f(u_2) = 1$ or $f(u_3) = 1$. 

\textbf{Case 1.1:} $f(u_2) = 1$. Then $f(u_3) \neq 1$. We claim $f(v_4) = 1$, since otherwise we color $u_4$ with $1$, and $u_1, u_5$ with $2$-colors (they have at least $2,2$ available $2$-colors). However, the number of available $2$-colors at $u_1, u_4, u_5$ is at least $3, 1, 3$. We are done by Hall's Theorem.

\textbf{Case 1.2:} $f(u_3) = 1$. Then $f(u_2) \neq 1$. We claim $f(u_7) = 1$, since otherwise we color $u_1$ with $1$, and $u_4, u_5$ with $2$-colors (they have at least $1,2$ available $2$-colors). Then $f(v_5) = 1$, since otherwise we color $u_5$ with $1$, and $u_1, u_4$ with $2$-colors (they have at least $2,1$ available $2$-colors). However, the number of available $2$-colors at $u_1, u_4, u_5$ is at least $3, 3, 2$. We are done by Hall's Theorem.
    


\textbf{Case 2:} None of $u_2, u_3, u_6, u_7$ is colored with $1$. By Case 1, each of $u_2, u_3, u_6, u_7$ sees color $1$ in its neighbourhood in $G$. We color $u_1$ with $1$. By Remark~\ref{saveonecolor}, $u_4, u_5$ have at least $2,2$ available $2$-colors. We are done by Hall's Theorem.  
\end{proof}

By Lemma~\ref{superbowtie}, a $5$-face cannot be adjacent to more than three $3$-faces and there is only one way it can be adjacent to three $3$-faces (See Figure~\ref{configurations}). We show it is also impossible.

\begin{lemma}\label{no5cycle3cyclebowtie}
A $5$-face can be adjacent to at most two $3$-faces.
\end{lemma}
\begin{proof}
Let $u_1u_2u_3u_4u_5u_1$ be a $5$-face. Suppose to the contrary that it is adjacent to three $3$-faces. By Lemma~\ref{superbowtie}, we may assume $u_1u_2u_6u_1, u_1u_5u_7u_1$, and $u_3u_4u_8u_3$ are three $3$-faces. By Lemma~\ref{superbowtie}, $u_6,u_7 \notin \{u_3, u_4, u_8\}$, $u_8 \notin \{u_2, u_5\}$, and $u_2u_8, u_5u_8 \notin E(G)$. By Lemma~\ref{two3cycles}, $u_6 \neq u_7$ and $u_2u_5,u_2u_7,u_5u_6, u_6u_7\notin E(G)$. By Lemma~\ref{3cycle4vertices}, $u_1,u_2,u_3,u_4,u_5,u_6,u_7,u_8$ are $4$-vertices. We delete $u_1$ and add $u_6u_7, u_2u_5$ to obtain a planar graph $G'$ with $\Delta(G')\le4$. By the minimality of $G$, $G'$ has a good coloring $f$. 

\textbf{Case 1:} All of $f(u_2), f(u_5), f(u_6), f(u_7)$ are not $1$. We color $u_1$ with $1$, and we are done.

\textbf{Case 2:} Two of $f(u_2), f(u_5), f(u_6), f(u_7)$ are $1$. Let $N(u_2)=\{u_1,u_3,u_6,v_2\}, N(u_3)=\{u_2,u_4,u_8,v_3\}, N(u_4)=\{u_3,u_5,u_8,v_4\}, N(u_5)=\{u_1,u_4,u_7,v_5\}, N(u_6)=\{u_1,u_2,v_6,w_6\}, N(u_7)=\{u_1,u_5,v_7,w_7\}$, and $N(u_8)=\{u_3,u_4,v_8,w_8\}$. By Lemma~\ref{two3cycles}, $v_3\ne v_4$, $v_2\notin\{v_6,w_6\}$ and $v_5\notin\{v_7,w_7\}$. By symmetry, we may assume $f(u_5) = f(u_6) = 1$. At most ten $2$-colors appear within distance two from $u_1$. We must have all ten $2$-colors appear exactly once at $u_2,v_2,u_3, u_4,v_5,v_6,w_6,u_7,v_7,w_7$, since otherwise we are done. Without loss of generality, we may assume $f(u_2)=2_1, f(u_3)=2_2, f(u_4)=2_3,f(u_7)=2_4,f(v_2)=2_5, f(v_5)=2_6, f(v_6)=2_7, f(w_6)=2_8, f(v_7)=2_9, f(w_7)=2_{10}$.

We claim $u_6u_8 \notin E(G)$, since otherwise, we uncolor $u_2$, and, by Remark~\ref{saveonecolor}, $u_1,u_2$ have at least $1,2$ available $2$-colors. We color them in order. Similarly, $u_7u_8 \notin E(G)$ (we need to switch the colors of $u_5$ and $u_7$ first). Note that either color $1$ appears in $N(v_6)$ or color $2_7$ appears in $N(w_6)$, since otherwise we can switch the colors of $u_6$ and $v_6$, recolor $u_2$ with $1$, and color $u_1$ with $2_1$. Similarly, we have either color $1$ appears in $N(w_6)$ or color $2_8$ appears in $N(v_6)$. Moreover, $2_3,2_6,2_9,2_{10}\in f(N(v_6)\cup N(w_6))$, since otherwise we can recolor $u_6$ with one of them, recolor $u_2$ with $1$ and color $u_1$ with $2_1$. Thus, we can switch the color of $u_2$ and $u_6$. Note that $u_3$ and $u_7$ can be at a distance of at most two, either due to $u_7u_3 \in E(G)$ or $u_7v_3 \in E(G)$. Similarly, $u_4$ and $u_6$ can be at a distance of at most two, either due to $u_6u_4 \in E(G)$ or $u_6v_4 \in E(G)$. We claim that at least one pair of $u_3,u_7$ and $u_4,u_6$ is at distance at least three. Suppose not, we either have (1) $u_7u_3 \in E(G)$ and $u_6u_4 \in E(G)$, or (2) $u_7u_3 \in E(G)$ and $u_6v_4 \in E(G)$, or (3) $u_7v_3 \in E(G)$ and $u_6u_4 \in E(G)$, or (4) $u_7v_3 \in E(G)$ and $u_6v_4 \in E(G)$. However, in each of the above-mentioned four cases, $u_1u_2u_3u_4u_5u_1$ is a separating $5$-cycle instead of a $5$-face.

Without loss of generality, we may assume that $u_7v_3\notin E(G)$. By Remark~\ref{saveonecolor}, each of $u_8,v_3$ and $v_4$ must see color $1$ in $G'-u'$ at least once. Note that color $2_1,2_5,2_7,2_8,2_9,2_{10}$ must be in $N[v_4]\cup N[u_8]\cup \{v_3\}$, since otherwise we can recolor $u_4$ with one of them and color $u_1$ with $2_3$. Similarly, color $2_4,2_6,2_7,2_8,2_9,2_{10}$ must be in $N[v_3]\cup N[u_8]\cup \{v_4\}$. Therefore, $\{f(u_8),f(v_3),f(v_4),f(v_8),f(w_8)\}=\{1,2_7,2_8,2_9,2_{10}\}$ and $1\notin\{f(v_3),f(v_4)\}$. Moreover, each of $v_6,w_6,v_7$ and $w_7$ must see color $1$ in $G'-u'$ at least once, since otherwise we recolor  $v_6,w_6,v_7$ or $w_7$ with $1$ and $u_1$ have an available $2$-color. And $2_2,2_5,2_7,2_8\in f(N(v_7)\cup N(w_7))$, since otherwise we can recolor $u_7$ with one of them and color $u_1$ with $2_4$. Recall that $2_3,2_6,2_9,2_{10}\in f(N(v_6)\cup N(w_6))$. Since the distance between $u_3u_7$ is at least three, we can switch the colors of $u_4$ and $u_7$, recolor $u_3$ with $2_3$ and color $u_1$ with $2_2$.

\textbf{Case 3:} Only one of $f(u_2), f(u_5), f(u_6), f(u_7)$ is $1$. By symmetry, we may assume $1\in\{f(u_2),f(u_6)\}$. By Case 2 (we are free to choose whether we want to switch the colors of $u_5$ and $u_7$ or $u_2$ and $u_6$), $u_5,u_7$ must see $1$. Thus, by Remark~\ref{saveonecolor}, $u_1$ always has one available $2$-color. We are done.
\end{proof}

By Lemma~\ref{3cycle4vertices} and~\ref{superbowtie}, if a $5$-face has a $3$-vertex then it can be adjacent to at most two $3$-faces. We show in two steps (Lemma~\ref{no5cycled3bowtie} and~\ref{no5cycle3d3cycle3cycle}) that it can be adjacent to at most one $3$-face. See Figure~\ref{configurations}.

\begin{lemma}\label{no5cycled3bowtie}
If a $5$-face has a $3$-vertex, then it cannot be adjacent to two $3$-faces that share a vertex.
\end{lemma}
\begin{proof}
Suppose not, i.e., there is such a structure in our graph $G$. Let $u_1u_2u_3u_4u_5u_1$ be a $5$-face with $u_4$ being a $3$-vertex and $u_1u_2u_6u_1$ and $u_1u_5u_7u_1$ be three $3$-faces. Let $N(u_4)=\{u_3, u_5, v_4\}$. Since $u_1u_2u_3u_4u_5u_1$ is a $5$-face, $u_6 \neq u_3$. By Lemma~\ref{3cycle4vertices} and~\ref{two3cycles}, $u_6 \notin \{u_4, u_5\}$, $u_7 \notin \{u_2, u_3, u_4\}$, $u_2u_5,u_6u_7 \notin E(G)$. By Lemma~\ref{3cycle4vertices} and~\ref{no33}, $u_1, u_2, u_3, u_5, u_6, u_7$ are all $4$-vertices. We delete $u_1$ and add $u_6u_7, u_2u_5$ to obtain a planar graph $G'$ with $\Delta(G') \le 4$. By the minimality of $G$, $G'$ has a good coloring $f$. 

\textbf{Case 1:} All of $f(u_2), f(u_5), f(u_6), f(u_7)$ are not $1$. We color $u_1$ with $1$, and we are done.

\textbf{Case 2:} Two of $f(u_2), f(u_5), f(u_6), f(u_7)$ are $1$. One of $u_2, u_6$ is $1$ and one of $u_5, u_7$ is $1$. By Remark~\ref{saveonecolor}, at most ten $2$-colors appear within distance two from $u_1$. Thus, all the ten $2$-colors must appear exactly once within distance two from $u_1$. We uncolor $u_4$. By Remark~\ref{saveonecolor}, $u_1, u_4$ have at least $1,2$ available $2$-colors. We are done by Hall's Theorem.  




\textbf{Case 3:} Only one of $f(u_2), f(u_5), f(u_6), f(u_7)$ is $1$. By Case 2 and Remark~\ref{saveonecolor}, $u_1$ always has one available $2$-color. We are done.
\end{proof}

\begin{lemma}\label{no5cycle3d3cycle3cycle}
If a $5$-face has a $3$-vertex, then it cannot be adjacent to two $3$-faces that do not share a vertex.
\end{lemma}
\begin{proof}
Suppose not, i.e., there is such a structure in $G$. Let $u_1u_2u_3u_4u_5u_1$ be a $5$-face, $u_1$ be a $3$-vertex, and $u_2u_3u_6u_2, u_4u_5u_7u_4$ be two $3$-faces. Let $N(u_1)=\{u_2, u_5, v_1\}$. By Lemma~\ref{3cycle4vertices}, $u_2u_5\notin E(G)$, $v_1\notin\{u_3, u_4, u_6, u_7\}$, and $d(u_2)=d(u_3)=d(u_4)=d(u_5)=d(u_6)=d(u_7)=4$. We delete $u_1$ and add $u_2u_5$ to obtain a planar graph $G'$ with $\Delta(G')\le4$. By the minimality of $G$, $G'$ has a good coloring $f$. We extend $f$ to $G$.

\textbf{Case 1:} $f(v_1)=1$ or $f(v_1)\notin\{f(u_2), f(u_5)\}$. By Remark~\ref{saveonecolor}, each of $u_2, u_5, v_1$ must see color $1$ at least once in $N[v_1], N[u_2], N[u_5]$, since otherwise we recolor $u_2, u_5, v_1$ with $1$. Therefore, at most nine $2$-colors appear within distance two from $u_1$, and we color $u_1$ with an available $2$-color.

\textbf{Case 2:} $f(v_1)\in\{f(u_2), f(u_5)\}$ and $f(v_1)\ne1$. By symmetry, we may assume that $f(v_1)=f(u_5)=2_1$. Let $N(u_2)=\{u_1,u_3,u_6,v_2\}, N(u_3)=\{u_2,u_4,u_6,v_3\}, N(u_4)=\{u_3,u_5,u_7,v_4\}, N(u_5)=\{u_1,u_4,u_7,v_5\}, N(u_6)=\{u_2,u_3,v_6,w_6\}$, and $N(u_7)=\{u_4,u_5,v_7,w_7\}$. Note that $1\in\{f(u_4), f(u_7), f(v_5)\}$, since otherwise we recolor $u_5$ with $1$ and it is already solved in Case 1. We uncolor $u_5$. We must have $f(u_2), f(u_3) \neq 1$, since otherwise, by Remark~\ref{saveonecolor}, at most nine $2$-colors appear within distance two from $u_5$, and we color $u_5$ with an available $2$-color. This is already solved in Case 1. Therefore, $f(u_2), f(u_3) \neq 1$ and all ten $2$-colors appear exactly once among $v_1, u_2, u_3, u_4, v_4, v_7, w_7, v_5$, and $N(v_5)$. We color $u_1$ with $1$. There are three subcases.

\textbf{Case 2.1:} $f(u_4)=1$. We uncolor $u_2$. By Remark~\ref{saveonecolor}, $u_2,u_5$ have at least $2,1$ available $2$-colors, and we color them in order.

 
\textbf{Case 2.2:} $f(u_7)=1$. We know $f(v_4) \ne 1$, since otherwise we can color $u_5$ with an available $2$-color, which is solved in Case 1. Thus, we can uncolor $u_7$ and recolor $u_4$ with $1$. By Remark~\ref{saveonecolor}, $u_7$ has at least one available $2$-color. We color $u_7$ with an available $2$-color, and it is already solved in Case 2.1.


\textbf{Case 2.3:} $f(v_5)=1$ and $1\notin\{f(u_4), f(u_7)\}$. By Case 2.1 and 2.2, $f(v_4) = 1$ and $1\in \{f(v_7),f(w_7)\}$. Therefore, $u_5$ has at least one available $2$-color, and we are done.
\end{proof}

By Lemma~\ref{no33}, a $5$-face can have at most two $3$-vertices. If it has two $3$-vertices, then we show that it cannot be adjacent to any $3$-faces. See Figure~\ref{configurations}.

\begin{lemma}\label{no5cycled3d33cycle}
If a $5$-face has two $3$-vertices, then it cannot be adjacent to any $3$-faces.
\end{lemma}

\begin{proof}
Suppose not, i.e., there is such a structure in our graph $G$. Let $u_1u_2u_3u_4u_5u_1$ be a $5$-face, $u_2, u_5$ be two $3$-vertices, and $u_3u_4u_6u_3$ be a $3$-face. Let $N(u_2)=\{u_1, u_3, v_2\}$ and $N(u_5)=\{u_1, u_4, v_5\}$. By Lemma~\ref{no33} and~\ref{3cycle4vertices}, $u_2u_5 \notin E(G)$ and $u_1,u_3,u_4, u_6$ are $4$-vertices. Let $N(u_1)=\{u_2,u_5,v_1,w_1\}, N(u_3)=\{u_2,u_4,u_6,v_3\}, \linebreak N(u_4)=\{u_3,u_5,u_6,v_4\}$, and $N(u_6)=\{u_3,u_4,v_6,w_6\}$. By Lemma~\ref{3cycle4vertices}, $u_1u_4\notin E(G)$. We delete $u_5$ and add $u_1u_4$ to obtain a planar graph $G'$ with $\Delta(G')\le4$. By the minimality of $G$, $G'$ has a good coloring $f$. 

\textbf{Case 1:} $f(v_5)=1$ or $f(v_5)\notin\{f(u_1),f(u_4)\}$. By Remark~\ref{saveonecolor}, $u_5$ has at least one available $2$-color, and we are done.

\textbf{Case 2:} $f(v_5)=f(u_4) \neq 1$. Then $1 \in \{f(u_3), f(u_6), f(v_4)\}$ since otherwise we recolor $u_4$ with $1$ and $f(v_5) \notin \{f(u_1), f(u_4)\}$. This is already solved in Case 1. We uncolor $u_4$. We claim $f(u_1), f(u_2) \neq 1$, since otherwise, by Remark~\ref{saveonecolor}, $u_4, u_5$ have at least $1,2$ available $2$-colors. We color them in order. Therefore, $f(u_1), f(u_2) \neq 1$, and all ten $2$-colors appear exactly once among $u_1, u_2, u_3, v_3, u_6, v_6, w_6, v_4, v_5$, and $N(v_4)$. We color $u_5$ with $1$. There are three subcases.

\textbf{Case 2.1:} $f(u_3)=1$. We uncolor $u_2$. By Remark~\ref{saveonecolor}, $u_2, u_4$ $,u_5$ have at least $2,1,3$ available $2$-colors, and we color them in the order $u_4,u_2,u_5$.

\textbf{Case 2.2:} $f(u_6)=1$. The proof is exactly the same as Case 2.1.


\textbf{Case 2.3:} $f(v_4)=1$, and $f(u_3), f(u_6)\neq 1$. By Case 2.1 and 2.2, $f(v_3) = 1$ and $1 \in \{f(v_6), f(w_6)\}$. Therefore, by Remark~\ref{saveonecolor}, $u_4,u_5$ have at least $2,1$ available $2$-colors, and we color them in the order $u_5,u_4$.


\textbf{Case 3:} $f(v_5)=f(u_1)\ne 1$. We uncolor $u_1,u_2,u_3$. We must have $1 \in \{f(v_1), f(w_1)\}$, since otherwise we color $u_1$ with $1$, and by Remark~\ref{saveonecolor}, $u_2,u_3,u_5$ have at least $2, 2, 3$ available $2$-colors. We are done by Hall's Theorem. Without loss of generality, say $f(v_1) = 1$. By Remark~\ref{saveonecolor}, we may assume $1 \in \{f(u_4), f(v_4), f(u_6)\}$, and $u_1, u_2,u_3,u_5$ have at least $1, 3, 2, 4$ available $2$-colors. We are done by Hall's Theorem.
\end{proof}

\textbf{Proof of Theorem~\ref{maintheorem1}:} By Euler's Formula for planar graphs, we have

\begin{equation}\label{discharging}
\sum\limits_{v \in V(G)} (d(v) - 4) +  \sum\limits_{f \in F(G)} (d(f) - 4) = -8.  
\end{equation}

Let the initial charge of each vertex and face be defined as $Ch(v) = d(v) - 4$ and $Ch(f) = d(f) - 4$. We define a few rules to redistribute the charges so that every vertex and face has a final charge of at least $0$.

(R1) Each $3$-vertex receives $\frac{1}{2}$ from each $5^+$-face it belongs to.

(R2) Each $3$-face receives $\frac{1}{2}$ from each adjacent $5^+$-face.

By Lemma~\ref{mindegree}, we only need to consider $3$-vertices and $4$-vertices. Since each $4$-vertex does not give or receive any charges, its charge remains $0$ until the end. By Lemma~\ref{3cycle4vertices}, a $3$-vertex cannot belong to any $3$-faces. By Lemma~\ref{3vertex4cycle}, a $3$-vertex can belong to at most one $4$-face. Therefore, it belongs to at least two $5^+$-faces and its final charge is at least $3-4+ 2 \cdot \frac{1}{2} = 0$. Every vertex has a final charge of at least $0$.

It remains to show every face has a final charge of at least $0$. A $4$-face has final charge $0$, since it does not give or receive any charges. By Lemma~\ref{two3cycles} and~\ref{no434}, a $3$-face cannot be adjacent to other $3$-faces and can be adjacent to at most one $4$-face. Therefore, it is adjacent to at least two $5^+$-faces and its final charge is at least $3-4+ 2 \cdot \frac{1}{2} = 0$. Let $f$ be a $5$-face. By Lemma~\ref{no33}, $f$ can have at most two $3$-vertices. If $f$ has two $3$-vertices, then by Lemma~\ref{3cycle4vertices} and~\ref{no5cycled3d33cycle}, it cannot be adjacent to any $3$-faces, and its final charge is at least $5 - 4 - 2 \cdot \frac{1}{2} = 0$. If $f$ has exactly one $3$-vertex, then by Lemma~\ref{3cycle4vertices},~\ref{no5cycled3bowtie}, and~\ref{no5cycle3d3cycle3cycle}, $f$ can be adjacent to at most one $3$-face, and its final charge is at least $5 - 4 - 2 \cdot \frac{1}{2} = 0$. Otherwise, $f$ has no $3$-vertices. By Lemma~\ref{superbowtie} and~\ref{no5cycle3cyclebowtie}, $f$ can be adjacent to at most two $3$-faces, and its final charge is at least $5 - 4 - 2 \cdot \frac{1}{2} = 0$. By Lemma~\ref{no33}, a $6$-face $f$ can have at most three $3$-vertices. If $f$ has three $3$-vertices, then, by Lemma~\ref{3cycle4vertices}, it cannot be adjacent to $3$-faces, and its final charge is at least $6-4-3 \cdot \frac{1}{2} \ge 0$. If $f$ has two $3$-vertices, then, by Lemma~\ref{3cycle4vertices}, it can be adjacent to at most two $3$-faces, and its final charge is at least $6-4-4 \cdot \frac{1}{2} \ge 0$. If $f$ has one $3$-vertex, then, by Lemma~\ref{3cycle4vertices} and~\ref{superbowtie}, it can be adjacent to at most three $3$-faces, and its final charge is at least $6-4-4 \cdot \frac{1}{2} \ge 0$. Otherwise, $f$ has no $3$-vertices. By Lemma~\ref{superbowtie}, it can be adjacent to at most four $3$-faces, and its final charge is at least $6-4-4 \cdot \frac{1}{2} \ge 0$.

For an $\ell$-face $f$, where $\ell \ge 7$, let $a$ be the number of $3$-vertices in $V(f)$ and $b$ be the number of $3$-faces adjacent to $f$. By Lemma~\ref{no33}, $2a \le \ell$. By Lemma~\ref{superbowtie}, $\frac{3}{2} b \le \ell$. By Lemma~\ref{3cycle4vertices}, $2a + b \le \ell$. Therefore, $\frac{a+b}{2} \le \frac{\ell + b}{4}$ and $\frac{b}{4} \le \frac{\ell}{6}$. We conclude $f$ has a final charge of at least

$$ \ell - 4 - \frac{1}{2} a - \frac{1}{2} b \ge \ell - 4 - \frac{\ell+b}{4} = \frac{3 \ell}{4} - 4 - \frac{b}{4} \ge \frac{7}{12} \ell - 4 \ge 0,$$
since $\ell \ge 7$. This is a contradiction with~\eqref{discharging}, and the proof is completed.

\section{Packing $(1^2,2^7)$-coloring of planar graphs of maximum degree at most $4$}
In this section, we show that every planar graph with maximum degree at most four has a packing $(1^2,2^7)$-coloring. We call a packing $(1^2,2^7)$-coloring {\em good coloring} in this section. We use $\{1_a,1_b\}$ to denote the $1$-colors, and let $\{2_1,2_2,2_3,2_4,2_5,2_6,2_7\}$ be the set of $2$-colors. Furthermore, in our reducibility proofs, whenever we delete a set of vertices $S$ from $G$ to obtain a smaller graph $G'$, we define the boundary vertices (denoted by $B$) as the set of all vertices in $G - S$ that are within distance two from $S$ in the original graph $G$.

\textbf{Proof of Theorem~\ref{maintheorem2}:} Suppose not, i.e., there exist planar graphs with maximum degree at most four that have no packing $(1^2,2^7)$-coloring. We take such a graph $G$ with minimum $|V(G)|+|E(G)|$.

\begin{lemma}\label{mindegree-2}
$\delta(G) \ge 3$.    
\end{lemma}

\begin{proof}
We first show there is no $1$-vertex. Suppose not, i.e., $u$ is a $1$-vertex in $G$ with its unique neighbour $u_1$. We delete $u$ from $G$ to obtain a planar graph $G'$ with $\Delta(G')\le4$. By the minimality of $G$, $G'$ had a good coloring $f$. Since there are at most $4$ vertices within distance two from $u$ and we have in total $9$ colors, we can extend $f$ to $G$. Then we show $\delta(G)\ge3$. Suppose not, i.e., $u$ is a $2$-vertex with neighbours $u_1, u_2$.  We delete $u$ from $G$ and add the edge $u_1u_2$ (if it is not already in $G$) to obtain a planar graph $G'$ with $\Delta(G') \le 4$. By the minimality of $G$, $G'$ has a good coloring $f$. Since there are at most $8$ vertices within distance two from $u$ and we have in total $9$ colors, we can extend $f$ to $G$. Note that $f$ has no $2$-conflict. 
\end{proof}

\begin{rem}\label{saveonecolor-2}
We use the following color-saving argument very often. Let $G$ be a graph, $f$ be a partial good coloring, and $u,v$ be two vertices with $uv \in E(G)$ such that $u,v$ have been assigned a color by $f$.

(i) If $f(N(u)) \cap \{1_a, 1_b\} \neq \emptyset$, then there are at least two $1$-colors appear in $N[u]$, since we recolor $u$ with a $1$-color if $|f(N(u)) \cap \{1_a, 1_b\}| = 1$. Otherwise, both $1_a, 1_b$ are available at $u$. 

(ii) $u,v$ together see at least two $1$-colors. If $N(u) \cup N(v) - \{u,v\}$ does not have $1$-color, then we recolor $u,v$ with $1_a, 1_b$. Otherwise, say $N(u)-v$ has a $1$-color, then we are done by (i).



\end{rem}

\begin{lemma}\label{no33-2}
Two $3$-vertices cannot be adjacent.
\end{lemma}

\begin{proof}
Suppose not, i.e., $u,v$ are two $3$-vertices with $uv\in E(G)$. Let $N(u)=\{v,u_1,u_2\}$ and $N(v)=\{u,v_1,v_2\}$. If $|\{u_1, u_2, v_1, v_2\}| = 2$, say $u_1 = v_1$ and $u_2 = v_2$, then we delete the edge $uv$ to obtain a planar graph $G'$ with $\Delta(G') \le 4$. By the minimality of $G$, $G'$ has a good coloring $f$. We uncolor $u,v$. Since we have in total $9$ colors, there are at least $3,3$ available colors at $u,v$, and we color them in order. If $|\{u_1, u_2, v_1, v_2\}| = 3$, say $u_1 = v_1$ and $u_2 \neq v_2$, we identify $u$ and $v$ (call this new vertex $u'$) to obtain a planar graph $G'$ with $\Delta(G') \le 4$. By the minimality of $G$, $G'$ has a good coloring $f$. If $f(u') \in \{1_a, 1_b\}$, say $f(u') = 1_a$, then, since we must see $1_b$ in both $N[u_1]$ and $N[v_2]$, we color $u$ with $1_a$ and $v$ with an available $2$-color. Otherwise, $f(u') \notin \{1_a, 1_b\}$. We must have $\{f(u_1), f(v_2) \} = \{1_a, 1_b\}$, say $f(u_1) = 1_a$ and $f(v_2) = 1_b$, since otherwise we color $u$ with $f(u')$ and $v$ with an available $1$-color. Furthermore, $u_1$ must see $1_b$, since otherwise we recolor $u_1$ with $1_b$. Similarly, $u_2$ must see $1_a$. We color $u,v$ with $f(u')$ and an available $2$-color. Thus, $|\{u_1, u_2, v_1, v_2\}| = 4$. 

We identify $u$ and $v$ (call this new vertex $u'$) to obtain a planar graph $G'$ with $\Delta(G') \le 4$. By the minimality of $G$, $G'$ has a good coloring $f$. We extend $f$ to $G$.

\textbf{Case 1:} $f(u')$ is a $1$-color, say $1_a$. Then $1_b \in \{f(u_1), f(u_2)\} \cap \{f(v_1), f(v_2)\}$, since otherwise we color one of $u,v$ with $1_b$ and the other with $1_a$. We may assume $f(u_1)=f(v_1)=1_b$. If $f(v_2) = 1_b$, then we can color $u,v$ with $1_a$ and a $2$-color unless all seven $2$-colors appear on $u_2, N(v_1)$, and $N(v_2)$. However, we can recolor both $v_1,v_2$ with $1_a$, and color $u,v$ with $1_a, 1_b$. Therefore, $f(v_2) \neq 1_b$ and $v_2$ must see $1_b$. By Remark~\ref{saveonecolor-2} (i), $\{1_a, 1_b\} \subseteq f(N[v_2])$ and we can color $u,v$ with $1_a$ and an available $2$-color.


\textbf{Case 2:} $f(u')$ is a $2$-color. Then $\{f(v_1), f(v_2)\} = \{f(u_1), f(u_2)\} = \{1_a, 1_b\}$, since otherwise we color $u,v$ with $f(u')$ and an available $1$-color. By Remark~\ref{saveonecolor-2} (i), we can color $u,v$ with $f(u')$ and an available $1$-color unless $\{1_a, 1_b\} \subseteq f(N[v_1])$. We color $u,v$ with $f(u')$ and an available $2$-color. This is a contradiction.
\end{proof}

\begin{lemma}\label{3cycle4vertices-2}
Every vertex of a $3$-cycle must be a $4$-vertex.    
\end{lemma}

\begin{proof}
Suppose $u_1u_2u_3u_1$ is a $3$-cycle and $u_1$ is a $3$-vertex. By Lemma~\ref{no33-2}, $u_2,u_3$ are $4$-vertices. Let $N(u_1) = \{v_1, u_2, u_3\}$ and $N(u_2)=\{u_1, u_3, v_2, w_2\}$. We delete $u_1$ and add $v_1u_2$ (if it is not already in $G$) to obtain a planar graph $G'$ with $\Delta(G') \le 4$. By the minimality of $G$, $G'$ has a good coloring $f$. 

\textbf{Case 1:} At most one of $f(v_1), f(u_2)$ and $f(u_3)$ is $1$-color. We can color $u_1$ with an available $1$-color.

\textbf{Case 2:} All of $f(v_1),f(u_2)$ and $f(u_3)$ are $1$-colors, say $f(v_1)=f(u_3)=1_a$ and $f(u_2)=1_b$. All seven $2$-colors must appear exactly once at $v_2, w_2$, $N(u_3) - u_2$, and $N(v_1)$, since otherwise we have an available $2$-color to use at $u_1$. We must have $1_a \in f(N(v_2))$ since otherwise we can recolor $v_2$ with $1_a$. Similarly, $1_a \in f(N(w_2))$. We uncolor $u_2$ and color $u_1$ with $1_b$. By Remark~\ref{saveonecolor-2} (i), $v_2,w_2$ must both see $1_b$. We color $u_2$ with an available $2$-color.


\textbf{Case 3:} Two of $f(v_1),f(u_2)$ and $f(u_3)$ are $1$-colors. Note that $\{1_a, 1_b\} \subseteq \{f(v_1), f(u_2), f(u_3)\}$, since otherwise we color $u_1$ with an available $1$-color. There are three subcases.

\textbf{Case 3.1:} $f(v_1), f(u_2) \in \{1_a, 1_b\}$. Say $f(v_1) = 1_a$ and $f(u_2) = 1_b$. Then $v_1$ must see $1_b$ in $N(v_1)$ and $u_2$ must see $1_a$ in $N(u_2)$, since otherwise we recolor $v_1$ to $1_b$ or $u_2$ to $1_a$, and color $u_1$ with an available $1$-color. Then $u_1$ has at least one available $2$-color and we are done.


\textbf{Case 3.2:} $f(u_2), f(u_3) \in \{1_a, 1_b\}$. Say $f(u_2) = 1_a$ and $f(u_3) = 1_b$. By Case 2, $v_1$ must see $1_b$. By Remark~\ref{saveonecolor-2} (i), $\{1_a, 1_b\} \subseteq f(N[v_1])$. Therefore, $u_1$ has at least one available $2$-color and we are done.

\textbf{Case 3.3:} $f(v_1), f(u_3) \in \{1_a, 1_b\}$. The proof is exactly the same as that of Case 3.1.
\end{proof}

\begin{lemma}\label{3vertex4cycle-2}
A $3$-vertex can belong to at most one $4$-face.
\end{lemma}
\begin{proof}
Suppose $u_1u_2u_3u_4u_1$ and $u_1u_4u_5u_6u_1$ are two $4$-faces sharing the edge $u_1u_4$ such that $u_1$ is a $3$-vertex. We claim that $u_2 \neq u_6$ and $u_3 \neq u_5$, since otherwise one of $u_1u_2u_3u_4u_1$ and $u_1u_4u_5u_6u_1$ is a separating cycle, and thus is not a $4$-face. By Lemma~\ref{3cycle4vertices-2}, $u_2u_6 \notin E(G)$, $u_2 \neq u_5$, and $u_3 \neq u_6$. We delete $u_1$ from $G$ and add the edge $u_2u_6$ to obtain a planar graph $G'$ with $\Delta(G') \le 4$. By the minimality of $G$, $G'$ has a good coloring of $f$. $1_a,1_b\in\{f(u_2), f(u_4), f(u_6)\}$, since otherwise we can color $u_1$ with an available $1$-color.

\textbf{Case 1:} $1_a,1_b\in\{f(u_2), f(u_6)\}$. Say, $f(u_2)=1_a$ and $f(u_6)=1_b$. If $f(u_4) \in \{1_a, 1_b\}$, then all seven $2$-colors must appear at $N(u_4) - u_1$, $N(u_2) - \{u_1,u_3\}$, and $N(u_6) - \{u_1, u_5\}$. However, we recolor $u_2,u_4,u_6$ to $1_a$, and color $u_1$ with $1_b$. Otherwise, $f(u_4) \notin \{1_a, 1_b\}$ and $u_4$ must see $1_a, 1_b$. Therefore, $u_1$ has at least one available $2$-color, and we are done.

\textbf{Case 2:} $1_a,1_b\in\{f(u_2),f(u_4)\}$. Say, $f(u_2)=1_a$ and $f(u_4)=1_b$. By Case 1 and $u_2u_6 \in E(G')$, $f(u_6) \notin \{1_a, 1_b\}$ and $u_6$ must see $1_b$. Furthermore, $u_4$ must see $1_a$, since otherwise we recolor $u_4$ with $1_a$ and color $u_1$ with $1_b$. Similarly, $u_2$ must see $1_b$. Thus, $u_1$ has at least one available $2$-color, and we are done.
\end{proof}

\begin{lemma}\label{no333-2} 
A $3$-cycle can share edges with $3$-cycle via at most one edge.
\end{lemma}

\begin{proof} 
Let $u_1u_2u_3u_1$ and $u_1u_3u_4u_1$ be two $3$-cycles such that they share the edge $u_1u_3$. By Lemma~\ref{3cycle4vertices-2} (i), $u_1,u_2,u_3,u_4$ are all $4$-vertices. Let $N(u_1)=\{u_2,u_3,u_4,u_5\}$. By symmetry, we need to show that $u_2u_4, u_4u_5\notin E(G)$. Suppose not, i.e., $u_4$ is adjacent to $u_2$ or $u_5$.
    
\textbf{Case 1:} $u_2u_4\in E(G)$. Let $N(u_3)=\{u_1,u_2,u_4,v_3\}$. We identify $u_1$ and $u_3$ (call this new vertex $u'$) to obtain a planar graph $G'$ with $\Delta(G)\le4$. By the minimality of $G$, $G'$ has a good coloring $f$. If $v_3 = u_5$, then, by Remark~\ref{saveonecolor-2} (i), each of $u_4,u_5$ sees a $1$-color. Therefore, $u_1, u_3$ have at least $2,2$ available $2$-colors, and we are done. Thus, we may assume $v_3 \neq u_5$. If $f(u_5)$ is a $1$-color, say $1_b$, then, by Remark~\ref{saveonecolor-2} (ii), $u_2,u_4$ together see at least two $1$-colors. By Remark~\ref{saveonecolor-2} (i), $v_3$ sees a $1$-color. Therefore, $u_1, u_3$ have at least $1,2$ available $2$-colors. We color them in order. Therefore, we may assume $f(u_3), f(u_5) \notin \{1_a,1_b\}$, and $u_3,u_5$ both see $1_a, 1_b$. By Remark~\ref{saveonecolor-2} (ii), $u_2,u_4$ together see at least two $1$-colors. Therefore, $u_1, u_3$ have at least $2,2$ available $2$-colors. We color them in order.




\textbf{Case 2:} $u_4u_5\in E(G)$. We delete $u_1$ and add $u_2u_5$ (if it is not already in $G$) to obtain a planar graph $G'$ with $\Delta(G')\le4$. By the minimality of $G$, $G'$ has a good coloring $f$. Note that $1_a,1_b\in\{f(u_2),f(u_3),f(u_4),f(u_5)\}$, since otherwise we can color $u_1$ with an available $1$-color.
    
\textbf{Case 2.1:} $1_a,1_b\in\{f(u_2),f(u_5)\}$, say $f(u_2)=1_a, f(u_5)=1_b$. Then $u_3$ sees $1_b$ since otherwise we recolor $u_3$ with $1_b$. Similarly, $u_4$ sees $1_a$. Therefore, we can color $u_1$ with an available $2$-color.
  
\textbf{Case 2.2:} $1_a,1_b\in\{f(u_2),f(u_3)\}$, say $f(u_2)=1_a, f(u_3)=1_b$. Note that $f(u_4) \neq 1_b$ and $u_5$ sees $1_b$ since otherwise we recolor $u_5$ with $1_b$. Similarly, $u_4$ sees $1_a$. Therefore, we can color $u_1$ with an available $2$-color.

\textbf{Case 2.3:} $1_a,1_b\in\{f(u_3),f(u_4)\}$, say $f(u_3)=1_a, f(u_4)=1_b$. Then $u_2$ sees $1_b$, since otherwise we recolor $u_2$ with $1_b$. Similarly, $u_5$ sees $1_a$. Therefore, we can color $u_1$ with an available $2$-color.

\textbf{Case 2.4:} $1_a,1_b\in\{f(u_2),f(u_4)\}$, say $f(u_2)=1_a, f(u_4)=1_b$. Since $f(u_3), f(u_5) \notin \{1_a, 1_b\}$, $u_4$ must see $1_a$, since otherwise we recolor $u_4$ with $1_a$ and color $u_1$ with $1_b$. We also know $u_5$ must see $1_a$, since otherwise we recolor $u_5$ with $1_a$. Therefore, we can color $u_1$ with an available $2$-color.
\end{proof}

\begin{lemma}\label{no334-2}
If a $3$-face is sharing an edge with another $3$-face, then it cannot share an edge with a $4$-face.
\end{lemma}
\begin{proof}
Let $u_1u_2u_3u_1$ and $u_1u_2u_4u_1$ be two $3$-faces sharing the edge $u_1u_2$. By Lemma~\ref{3cycle4vertices-2}, $u_1,u_2,u_3,u_4$ are all $4$-vertices. The face $u_1u_2u_3u_1$ cannot share the edge $u_1u_2$ with another $4$-face, since otherwise this $4$-face, $u_1u_2u_3u_1$, or $u_1u_2u_4u_1$ is a separating cycle, which contradicts they are faces. Suppose $u_1u_2u_3u_1$ shares the edge $u_1u_3$ with a $4$-face $u_1u_5u_6u_3u_1$. We know $u_4 \neq u_5$ and $u_2 \neq u_6$, since otherwise $u_1u_5u_6u_3u_1$, $u_1u_2u_3u_1$, or $u_1u_2u_4u_1$ is a separating cycle, which contradicts they are faces. By Lemma~\ref{no333-2}, $u_2, u_4 \neq u_6$ and $u_2 \neq u_5$. Therefore, $|\{u_1, u_2, u_3, u_4, u_5, u_6\}| = 6$. Let $N(u_2)=\{u_1,u_3,u_4,v_2\}, N(u_3)=\{u_1,u_2,u_6,v_3\}$. By Lemma~\ref{no333-2}, $u_4u_5, u_3u_4, u_2u_6, u_3u_5 \notin E(G)$. We identify $u_1$ and $u_2$ (call this new vertex $u'$) to obtain a planar graph $G'$ with $\Delta(G)\le4$. By the minimality of $G$, $G'$ has a good coloring $f$. 

We first show $u_2u_5 \notin E(G)$. Suppose not, i.e., $u_2u_5 \in E(G)$. We color $u_2$ with $f(u')$. If $f(u')$ is a $1$-color, say $1_a$, then, by Remark~\ref{saveonecolor-2}, each of $u_4, u_5$ sees $1_b$. Therefore, we can color $u_1$ with an available $2$-color. Otherwise, $f(u')$ is a $2$-color. By applying Remark~\ref{saveonecolor-2} on $u_4, u_5$, we know each of $u_4, u_5$ sees a $1$-color. We claim $|\{f(u_3), f(v_3)\} \cap \{1_a, 1_b\}| \ge 1$, since otherwise we recolor $u_3$ with an available $1$-color. Therefore,  we can color $u_1$ with an available $2$-color.

\textbf{Case 1:} $f(u')$ is a $1$-color, say $f(u')=1_a$. Each of $u_1,u_2$ must see $1_b$ at least once, since otherwise we color one of them with $1_b$ and the other with $1_a$. If $f(v_2)=1_b$, then we color $u_2$ with $1_a$. By Remark~\ref{saveonecolor-2}, each of $u_3, u_4$ sees $1_b$ and $u_5$ sees $1_a$. We color $u_1$ with an available $2$-color. Therefore, we may assume $f(v_2)$ is a $2$-color and $v_2$ sees $1_b$. If $f(u_5)=1_b$, then we color $u_1$ with $1_a$. By Remark~\ref{saveonecolor-2}, each of $u_3, u_4$ see $1_b$ and $v_2$ sees $1_a$. We can color $u_2$ with an available $2$-color. Therefore, we may also assume $f(u_5)$ is a $2$-color and $u_5$ sees $1_b$. We claim $f(u_6) = 1_b$, since otherwise, we color $u_2$ with $1_a$, and by Remark~\ref{saveonecolor-2}, each of $u_3, u_4, u_5$ see $1_b$. Thus, we can color $u_1$ with an available $2$-color. Therefore, we can assume $f(u_6)=1_b$. This implies $f(u_3) \neq 1_b$ and thus $f(u_4) = 1_b$. By Remark~\ref{saveonecolor-2}, $u_5$ sees $1_a$. We further claim that $u_4$ sees $1_a$, since otherwise, we can recolor $u_4$ with $1_a$ and color $u_2$ with $1_b$. We may assume $1_a \in \{f(u_3), f(v_3)\}$, since otherwise we recolor $u_3$ with $1_a$. We color $u_1$ with an available $2$-color. Therefore, $u_4$ must see $1_a$.  We color $u_2$ with $1_a$ and at most six $2$-colors appear within distance two from $u_1$. We color $u_1$ with an available $2$-color.
    
\textbf{Case 2:} $f(u')$ is a $2$-color. Then $1_a,1_b \in \{f(u_3),f(u_4),f(u_5)\}$, since otherwise we can color $u_1$ with an available $1$-color and $u_2$ with $f(u')$. Similarly, $1_a,1_b \in \{f(u_3),f(u_4),f(v_2)\}$. We color $u_2$ with $f(u')$.

\textbf{Case 2.1:} $f(u_3)$ is a $1$-color, say $1_a$. Then $f(u_4)$ must be $1_b$. Otherwise, $f(u_5) = f(v_2) = 1_b$. By Remark~\ref{saveonecolor-2} and $f(u_6) \neq 1_a$, each of $u_4$ and $u_5$ sees $1_a$. We color $u_1$ with an available $2$-color. Thus, we may assume $f(u_4) = 1_b$ and $u_4$ sees $1_a$. If $f(u_5) = 1_b$, then, since $u_3$ must see $1_b$, we color $u_1$ with an available $2$-color. Hence, $f(u_5) \neq 1_b$ and $u_5$ sees $1_b$. By Remark~\ref{saveonecolor-2}, $u_5$ sees $1_a$. We color $u_1$ with an available $2$-color.

\textbf{Case 2.2:} $f(u_3)$ is a $2$-color. Then $\{f(u_4), f(u_5)\} = \{1_a, 1_b\}$, say $f(u_4) = 1_a$ and $f(u_5) = 1_b$. Furthermore, $f(v_2) = 1_b$. By Case 2.1, $u_3$ sees $1_a, 1_b$. Since $f(u_6) \neq 1_b$, $f(v_3) = 1_b$ and $f(u_6) = 1_a$. Therefore, at most six $2$-colors appear within distance two from $u_1$. We color $u_1$ with an available $2$-color.
\end{proof}

\begin{lemma}\label{bowtie-2}
Two $3$-cycles cannot share a vertex without sharing an edge.
\end{lemma}

\begin{proof} 
Suppose $u_1u_2u_3u_1$ is a $3$-cycle. By Lemma~\ref{3cycle4vertices-2}, $u_1,u_2,u_3$ are $4$-vertices. Let $N(u_1)=\{u_2,u_3,u_4,u_5\}$ and $ N(u_2)=\{u_1,u_3,v_2,w_2\}$. In order to prove the lemma, we need to show that $u_4u_5\notin E(G)$. Suppose not, i.e., $u_4u_5\in E(G)$. By Lemma~\ref{no333-2}, $u_2u_4, u_2u_5, u_3u_4, u_3u_5\notin E(G)$. We delete $u_1$ and add $u_2u_4, u_3u_5$ to obtain a planar graph $G'$ with $\Delta(G')\le4$. By the minimality of $G$, $G'$ has a good coloring $f$. Note that $\{1_a,1_b\}\subseteq\{f(u_2),f(u_3),f(u_4),f(u_5)\}$, since otherwise we can color $u_1$ with an available $1$-color. By Lemma~\ref{3cycle4vertices-2}, $u_4,u_5$ are also $4$-vertices. Let $N(u_i) - \{u_1, u_2, u_3, u_4, u_5\} = \{v_i, w_i\}$, where $i \in \{2,3,4,5\}$.

\textbf{Case 1:} $u_2$ and $u_4$ share a neighbour in $G-u_1$, say $v_2 = v_4$. If $f(v_2)\in\{1_a,1_b\}$, say $f(v_2)=1_a$, then $f(u_2), f(u_4) \neq 1_a$. We claim each of $u_3$ and $u_5$ must see $1_a$, since otherwise we recolor $u_3, u_5$ with $1_a$. Furthermore, we claim each of $u_2$ and $u_4$ must see $1_b$, since otherwise we recolor $u_2,u_4$ with $1_b$. However, $u_1$ has at least one available $2$-color, and we are done. Therefore, we can assume $f(v_2)$ is a $2$-color. Then we claim $f(w_2) \notin \{1_a, 1_b\}$. Suppose not, i.e., $f(w_2) = 1_a$. Thus, $f(u_2) \neq 1_a$. Then $u_3$ must see $1_a$, since otherwise we recolor $u_3$ with $1_a$. Furthermore, $u_2$ must see $1_b$, since otherwise we recolor $u_2$ with $1_b$. Therefore, at most two of $w_2, u_2, u_3, v_3, w_3$ have $2$-colors. By Remark~\ref{saveonecolor-2}, $u_4$ and $u_5$ see at most three $2$-colors. We can color $u$ with an available $2$-color. Similarly, $v_3, w_3, w_4, v_5, w_5$ cannot be a $1$-color. We uncolor $u_2,u_3,u_4,u_5$. Each of $v_2,w_2,v_3,w_3,w_4,v_5,w_5$ must see both $1_a$ and $1_b$, since otherwise we can recolor one of them with a $1$-color, color $u_2,u_3,u_4,u_5$ with available $1$-colors and $u_1$ with an available $2$-color. Therefore, $u_2,u_4$ has at least $2,2$ available $2$-colors, and we can color $u_2,u_4$ in order with $2$-colors, $u_3,u_5$ with $1_a$ and $u_1$ with $1_b$. 

    
\textbf{Case 2:} $u_2$ and $u_3$ share a neighbour in $G-u_1$, say $v_2 = v_3$. We claim $f(v_4) \notin \{1_a, 1_b\}$. Suppose not, say $f(v_4) = 1_a$ and note $f(u_4) \neq 1_a$. Then $u_5$ must see $1_a$, since otherwise we recolor $u_5$ with $1_a$. Furthermore, $u_4$ must see $1_b$, since otherwise we recolor $u_4$ with $1_b$. Therefore, at most three of $u_4,v_4,w_4,u_5,v_5,w_5$ have a $2$-color. By Remark~\ref{saveonecolor-2}, at most three of $w_2,u_2,v_2,u_3,w_3$ have $2$-colors. We can color $u_1$ with an available $2$-color, and we are done. Similarly, $w_4, v_5, w_5, w_2, w_3$ have $2$-colors.


\textbf{Case 2.1:} $f(v_2)$ is a $1$-color, say $f(v_2)=1_a$. We uncolor $u_2,u_3,u_4,u_5$ and color $u_3,u_4,u_5$ with available $1$-colors. We claim that $w_2$ must see both $1_a$ and $1_b$. Suppose not, we can recolor $w_2$ with an available $1$-color. The number of available $2$-colors at $u_1,u_2$ are $2,1$. We color them in the order $u_2,u_1$, and we are done. Thus, the number of available $2$-colors at $u_1,u_2$ are $1,2$, and we color them in order.


\textbf{Case 2.2:} $f(v_2)$ is a $2$-color. Uncolor $u_2,u_3,u_4,u_5$. Each of $w_2,w_3,v_4,w_4,u_5,v_5$ must see both $1_a$ and $1_b$, since otherwise we can color one of them with a $1$-color, color $u_2,u_3,u_4,u_5$ with available $1$-colors and $u_1$ with an available $2$-color. Furthermore, $v_2$ must see both $1_a$ and $1_b$, since otherwise we can color $v_2,u_3,u_4,u_5$ with available $1$-colors and $u_1,u_2$ have $1,2$ available colors. We color them in the order $u_1,u_2$. Thus, the number of available $2$-colors at $u_3,u_4$ are $2,1$, we color them in the order $u_4,u_3$. Then we can color $u_2,u_5$ with $1_a$ and $u_1$ with $1_b$.

\textbf{Case 3:} Any two of $u_2,u_3,u_4,u_5$ do not share a neighbour in $G-u_1$. 

\textbf{Case 3.1:} All of $f(u_2), f(u_3), f(u_4), f(u_5)$ are $1$-colors, say $f(u_2)=f(u_5)=1_a$ and $f(u_3)=f(u_4)=1_b$. All seven $2$-colors must appear in $v_2,w_2,v_3,w_3,v_4,w_4,v_5,w_5$, since otherwise we can color $u_1$ with an available $2$-color. If $1_b \in \{f(v_2), f(w_2)\}$, say $f(v_2) = 1_b$, we uncolor $u_3,u_4$. Each of $w_2,v_3,w_3,v_4, w_4, v_5,w_5$ must see $1_b$ and each of $v_3,w_3,v_4,w_4$ must see $1_a$, since otherwise, we can recolor one of them with an available $1$-color and color $u_1$ with an available $2$-color. We color $u_1$ with $1_b$. The number of available $2$-colors at $u_3,u_4$ are at least $2,1$. We can color them in the order $u_4,u_3$, and we are done. Therefore, $f(v_2)$ must be a $2$-color, and $v_2$ sees both $1_a, 1_b$ in $N(v_2) - u_2$ (we need to switch the colors of $u_2$ and $u_3$ for the proof $v_2$ must see $1_a$ in $N(v_2) - u_2$). Similarly, each $f(v_i)$ ($f(w_i)$) are $2$-colors and $v_i$ ($w_i$) must see $1_a, 1_b$ in $N(v_i) - u_i$ ($N(w_i) - u_i$), where $2 \le i \le 5$.

We uncolor $u_3, u_4$. Note that exactly one of the $2$-colors appear twice. By symmetry, we may assume $v_2,w_2,v_3,w_3,v_4, w_4,v_5,w_5$ have colors $2_1, 2_2, 2_3, 2_7, 2_4, 2_5, 2_6, 2_7$ or $2_1, 2_2, 2_3, 2_1, 2_4, 2_5, 2_6, 2_7$. In the former case, $u_3$ has an available color from $\{2_4, 2_5, 2_6\}$ and $u_4$ has an available color from $\{2_1, 2_2, 2_3\}$. In the latter case, $u_3$ has an available color from $\{2_4, 2_5, 2_6, 2_7\}$ and $u_4$ has an available color from $\{2_1, 2_2, 2_3\}$. Therefore, we can always color $u_3, u_4$ with $2$-colors, and color $u_1$ with $1_b$. This is a contradiction.

\textbf{Case 3.2:} Three of $\{f(u_2),f(u_3),f(u_4),f(u_5)\}$ are $1$-colors, say $f(u_2)=1_a$, $f(u_3) = f(u_4) = 1_b$, and $f(u_5) = 2_1$. By Case 3.1, $u_5$ must see color $1_a$, say $f(v_5)=1_a$. We next show $f(w_5) = 1_b$. Suppose not, $f(w_5) \neq 1_b$. We claim $u_4$ must see $1_a$, since otherwise we can recolor $u_4, u_5$ with $1_a, 1_b$, which is solved in Case 3.1. Since $u_1$ must see all the seven $2$-colors, $f(w_5)$ must be a $2$-color and $w_5$ must see $1_a, 1_b$ (otherwise, we recolor $w_5$ with a color in $\{1_a, 1_b\}$). Furthermore, $v_5$ must see $1_b$ since otherwise we recolor $v_5, u_5$ with $1_b, 1_a$, and this is solved in Case 3.1. We uncolor $u_5$, and $u_1, u_5$ have at least $1, 2$ available $2$-colors. 

Thus, $f(w_5) = 1_b$. Since $u_1$ must see all the seven $2$-colors, we may assume $v_2,w_2,v_3,w_3,v_4,w_4$ are colored with $2_2, 2_3, 2_4, 2_5, 2_6, 2_7$. We know $v_5$ sees $1_b$ since otherwise we recolor $v_5, u_5$ with $1_b, 1_a$. Let $N(w_5) = \{u_5, x_5, y_5, z_5\}$ and . Then $w_5$ sees $1_a$, say $f(x_5) = 1_a$, since otherwise we recolor $w_5, u_5, u_4$ with $1_a, 1_b, 1_a$, and color $u_1$ with $2_1$. This is a contradiction. Furthermore, $u_5$ must see all of $2_2, 2_3, 2_4, 2_5, 2_6, 2_7$, since otherwise we recolor $u_5$ with an available $2$-color from $\{2_2, 2_3, 2_4, 2_5, 2_6, 2_7\}$ and color $u_1$ with $2_1$. Therefore, we may assume $v_5$ has $2_4, 2_5$ in its neigbours and $f(y_5) = 2_6$ and $f(z_5) = 2_7$. We can also assume $y_5,z_5$ both see $1_a$, since otherwise we can recolor $y_5$ or $z_5$ with $1_a$. We then claim $x_5$ sees $1_b$ in $N(x_5) - w_5$, since otherwise we recolor $x_5, w_5$ with $1_b, 1_a$, which is a contradiction with the fact that $f(w_5)$ must equal $1_b$. Moreover, we can switch the colors of $w_5$ and $y_5$ or the colors of $w_5$ and $z_5$, which implies either $1_b \in f(N(y_5) - w_5)$ or $2_6 \in f(N(x_5) - w_5) \cup f(N(z_5) - w_5)$ and either $1_b \in f(N(z_5) - w_5)$ or $2_7 \in f(N(x_5) - w_5) \cup f(N(y_5) - w_5)$. We uncolor $u_4,u_5$. However, we must be able to recolor $w_5$ with one of $2_1,2_2,2_3,2_4, 2_5$, and color $u_5, u_4$ with $1_b, 1_a$. This is already solved in Case 3.1.

\textbf{Case 3.3:} Two of $\{f(u_2),f(u_3),f(u_4),f(u_5)\}$ are $1$-colors. By symmetry, we may assume $f(u_2) = 1_a, f(u_3) = 1_b, f(u_4) = 2_1, f(u_5) = 2_2$ or $f(u_2) = 1_a, f(u_3) = 2_1, f(u_4) = 1_b, f(u_5) = 2_2$. In the former case, we know both $u_4, u_5$ see $1_a, 1_b$, since otherwise we can recolor $u_4$ or $u_5$ with a $1$-color and this is already solved in Case 3.2. In the latter case, we know from Case 3.2 that $u_3$ sees $1_b$ and $u_5$ sees $1_a$. Furthermore, $u_4$ sees $1_a$ since otherwise we recolor $u_4$ with $1_a$ and color $u_1$ with $1_b$. Similarly, $u_2$ sees $1_b$. Therefore, in both the former and the latter case, $u_1$ sees at most six $2$-colors, and we can color it with an available $2$-color. This is a contradiction.
\end{proof}


\begin{lemma}\label{no434-2}
A $3$-face can share edges with $4$-faces via at most one edge.
\end{lemma}

\begin{proof}
Let $u_1u_2u_3u_1$ be a $3$-face and $u_1u_2u_4u_5u_1$ be a $4$-face such that they share the edge $u_1u_2$. By Lemma~\ref{3cycle4vertices-2}, $u_1,u_2,u_3$ are all $4$-vertices. Let $N(u_1)=\{u_2,u_3,u_5,u_7\},N(u_2)=\{u_1,u_3,u_4,v_2\}$, $N(u_3)=\{u_1,u_2,u_6,v_3\}$. By Lemma~\ref{bowtie-2}, $v_2u_4,v_3u_6,u_5u_7\notin E(G)$. We need to prove $u_6u_7\notin E(G)$. Suppose not, i.e., $u_6u_7\in E(G)$. We delete $u_1,u_2,u_3$ and add $v_2u_4,v_3u_6,u_5u_7$ to obtain a planar graph $G'$ with $\Delta(G') \le 4$. By the minimality of $G$, $G'$ has a good coloring $f$. Let $N(u_4) = \{u_2,  u_5, v_4, w_4\}$, $N(u_5) = \{u_1, u_4, v_5, w_5\}$, $N(u_6) = \{u_3, u_7, v_6, w_6\}$, $N(u_7) = \{u_1, u_6, v_7, w_7\}$, $N(v_2) = \{u_2, x_2, y_2, z_2\}$, and $N(v_3) = \{u_3, x_3, y_3, z_3\}$. Let $B = \{u_4, u_5, u_6, u_7, v_2, v_3, v_4, w_4, v_5, w_5, v_6, w_6, v_7, w_7, x_2, y_2, z_2, x_3, y_3, z_3\}$ be the boundary vertices. We may assume $|B| = 20$, since if not, say $u_4,u_5$ have a common neighbour $v_4$ and each has a distinct neighbour $v_5, w_4$. We assume $v_4, v_5, w_4$ are colored with $f(v_4), f(w_4), f(v_5)$. It can be covered by the case when each of $u_4, u_5$ has two distinct neighbours $v_4,w_4, v_5, w_5$ and they are colored with $f(v_4), f(w_4), f(v_5), f(v_4)$. Our proof strategy is to extend every possible pre-coloring of $B$ in $G'$ to a good coloring of $G$. We use $C++$ programs (see Appendix for a detailed explanation) to check if each pre-coloring can be extended. It turns out that we cannot directly extend to $G$ on only $8$ cases (See Figure~\ref{434-badcases}). 
    
We solve those $8$ cases here. Since the argument is the same for Cases $1$, $2$, $3$, and $4$, and symmetric for Cases $1'$, $2'$, $3'$, and $4'$, we only show the argument for Case 1. We claim $w_4$ must see $1_b$, since otherwise, we can recolor $w_4$ with $1_b$, color $u_2$ with $2_5$, and color $u_1,u_3$ with available $1$-colors. Furthermore, $v_4$ must see $1_a$ in $N(v_4) - u_4$, since otherwise, we can switch the color of $v_4$ and $u_4$, and color $u_2$ with $1_a$, $u_3$ with $1_b$ and $u_1$ with an available $2$-color. We uncolor $u_4$. We claim $w_4$ must see $1_a$ in $N(w_4) - u_4$, since otherwise we recolor $w_4$ with $1_a$. Thus, the number of available $2$-colors at $u_1,u_4$ is at least $3,2$, we can color them in the order $u_4,u_1$ and color $u_2$ with $1_a$ and $u_3$ with $1_b$. This is a contradiction.
\end{proof}

\begin{lemma}\label{no5cycle3d3cycle3cycle-2}
If a $5$-face has a $3$-vertex, then it cannot be adjacent to two $3$-faces that do not share a vertex.
\end{lemma}

\begin{proof}
Suppose not, i.e., there is such a structure in $G$. Let $u_1u_2u_3u_4u_5u_1$ be a $5$-face, $u_1$ be a $3$-vertex, and $u_2u_3u_6u_2, u_4u_5u_7u_4$ be two $3$-faces. Let $N(u_1)=\{u_2, u_5, v_1\}$. By Lemma~\ref{3cycle4vertices-2}, $u_2u_5\notin E(G)$, $v_1\notin\{u_3, u_4, u_6, u_7\}$, and $d(u_2)=d(u_3)=d(u_4)=d(u_5)=d(u_6)=d(u_7)=4$. Let $N(u_2)=\{u_1,u_3,u_6,v_2\},N(u_3)=\{u_2,u_4,u_6,v_3\},N(u_4)=\{u_3,u_5,u_7,v_4\}$ and $N(u_5)=\{u_1,u_4,u_7,v_5\},$ $N(u_6)=\{u_2, u_3, v_6, w_6\},$ $ N(u_7)=\{u_4, u_5, v_7, w_7\},$ $N(v_1)=\{u_1, x_1, y_1, z_1\},$ $N(v_2)=\{u_2, x_2, y_2, z_2\},$ $N(v_3)=\{u_3, x_3, y_3, z_3\},$ $N(v_4)=\{u_4, x_4, y_4, z_4\}$ and $N(v_5)=\{u_5, x_5, y_5, z_5\}$. We delete $u_1,u_2,u_3,u_4,u_5$ and add $u_6v_2,u_6v_3,u_7v_4,u_7v_5$ to obtain a planar graph $G'$ with $\Delta(G')\le4$. By the minimality of $G$, $G'$ has a good coloring $f$. Let $B = \{u_6, u_7, v_1, v_2, v_3, v_4, v_5, v_6, w_6, v_7, w_7, x_1, \linebreak y_1, z_1, x_2, y_2, z_2, x_3, y_3, z_3, x_4, y_4, z_4, x_5, y_5, z_5\}$ be the boundary vertices. We may assume $|B|=26$, since if not, say $v_4,v_5$ have a common neighbour $x_4$ and each has two distinct neighbours $y_4,z_4,x_5,y_5$. We may assume $x_4,y_4,z_4,x_5,y_5$ are colored with $f(x_4),f(y_4),f(z_4),f(x_5),f(y_5)$. It can be covered by the case when each of $v_4,v_5$ has three distinct neighbours $x_4,y_4,z_4,x_5,y_5,z_5$ and they are colored with $f(x_4),f(y_4),f(z_4),\linebreak f(x_5),f(y_5),f(x_4)$. 
Our proof strategy is to extend every possible pre-coloring of $B$ in $G'$ to a good coloring of $G$.
We use $C++$ programs (see Appendix for a detailed explanation) to check if each pre-coloring can be extended. It turns out that we cannot directly extend to $G$ on only $11$ cases (See Figure~\ref{3d533-badcases}). 
In those cases, $u_6$ and $u_7$ are assigned a special marker, $-1$, representing an undetermined $1$-color. When we try to extend the pre-coloring, the configuration is considered extendable if replacing $-1$ with either $1_a$ or $1_b$ yields a valid extension.
    
We solve those $11$ cases here. Since the argument is the same for all of the $11$ cases, we only show the argument for Case 1. We color $u_5$ with $2_7$. $\{1_a,1_b\}\subseteq f(N(v_6))\cap f(N(w_6))$, since otherwise we recolor one of $v_6,w_6$ with an available $1$-color, $u_1, u_6, u_3, u_4, u_7$ with available $1$-colors and $u_2$ with $2_5$ or $2_6$. Thus, we can color $u_6$ with an available $2$-color and $u_1,u_2,u_3,u_4,u_7$ with available $1$-colors.
\end{proof}

\begin{lemma}\label{no5cycled3d33cycle-2}
If a $5$-face has two $3$-vertices, then it cannot be adjacent to any $3$-faces.
\end{lemma}

\begin{proof}
Suppose not, i.e., there is such a structure in our graph $G$. Let $u_1u_2u_3u_4u_5u_1$ be a $5$-face, $u_2, u_5$ be two $3$-vertices, and $u_3u_4u_6u_3$ be a $3$-face. By Lemma~\ref{no33-2} and~\ref{3cycle4vertices-2}, $u_2u_5,u_1u_3,u_1u_4 \notin E(G)$, $v_2\notin\{u_4,v_1,w_1,v_3,u_6\}$, $v_5\notin\{u_3,v_1,w_1,v_4,u_6\}$ and $u_1,u_3,u_4, u_6$ are $4$-vertices. Let $N(u_1)=\{u_2,u_5,v_1,w_1\}$, $N(u_2)=\{u_1, u_3, v_2\}$, $N(u_3)= \{u_2,u_4,u_6,v_3\}$, $N(u_4)=\{u_3,u_5,u_6,v_4\}$, $N(u_5)=\{u_1,u_4,v_5\}$, $N(u_6)=\{u_3,u_4,v_6,w_6\}$, $N(v_2)=\{u_2, x_2, y_2, z_2\}$, $N(v_3)=\{u_3, x_3, y_3, z_3\}$, $N(v_4)=\{u_4, x_4, y_4, z_4\}$ and $N(v_5)=\{u_5, x_5, y_5, z_5\}$. We delete $u_2,u_3,u_4,u_5$ and add $u_1v_2,u_1v_5,u_6v_3,u_6v_4$ to obtain a planar graph $G'$ with $\Delta(G')\le4$. By the minimality of $G$, $G'$ has a good coloring $f$. Let $B = \{u_1,u_6,v_2,v_3,v_4,v_5,v_1, w_1,v_6, w_6, x_2, y_2, z_2, x_3, y_3, z_3, x_4, y_4, z_4, x_5, y_5, z_5\}$ be the boundary vertices. We may assume $|B|=22$, since if not, say $v_4,v_5$ have a common neighbour $x_4$ and each has two distinct neighbours $y_4,z_4,x_5,y_5$. We may assume $x_4,y_4,z_4,x_5,y_5$ are colored with $f(x_4),f(y_4),f(z_4),f(x_5),f(y_5)$. It can be covered by the case when each of $v_4,v_5$ has three distinct neighbours $x_4,y_4,z_4,x_5,y_5,z_5$ and they are colored with $f(x_4),f(y_4),f(z_4),f(x_5),f(y_5),f(x_4)$. 
Our proof strategy is to extend every possible pre-coloring of $B$ in $G'$ to a good coloring of $G$. We use $C++$ programs (see Appendix for a detailed explanation) to check if each pre-coloring can be extended. It turns out that we cannot directly extend to $G$ on only $28$ cases (See Figures~\ref{3d533-badcases} and~\ref{3d3d53-badcases}). 
In those cases, some vertices are assigned a special marker, $-1$, representing an undetermined $1$-color. When we try to extend the pre-coloring, the configuration is considered extendable if replacing $-1$ with either $1_a$ or $1_b$ yields a valid extension.
    
We solve those $28$ cases here. Since the arguments for the first 11 cases are identical, and the same holds for the remaining 17 cases, we only present the proofs for Case 1-1 and Case 2-1.

In Case 1-1, each of $v_6,w_6$ must see a $1$-color, since otherwise we recolor $v_6$ or $w_6$ with a $1$-color. Thus, we can color $u_6$ with an available $2$-color and $u_2,u_3,u_4,u_5,v_3,v_4$ with available $1$-colors.

In Case 2-1, $\{1_a,1_b\}\subseteq f(N(v_6))\cap f(N(w_6))$, since otherwise we can recolor $v_6$ or $w_6$ with an available $1$-color, $u_3$ with $2_6$ or $2_7$, and $u_2,u_4,u_5,u_6$ with available $1$-colors. Thus, we can color $u_6$ with an available $2$-color and $u_2,u_3,u_4,u_5$ with available $1$-colors. This is a contradiction.
\end{proof}

\textbf{Proof of Theorem~\ref{maintheorem2}:} By Euler's Formula for planar graphs, we have

\begin{equation}\label{discharging2}
\sum\limits_{v \in V(G)} (d(v) - 4) +  \sum\limits_{f \in F(G)} (d(f) - 4) = -8.  
\end{equation}

Let the initial charge of each vertex and face be defined as $Ch(v) = d(v) - 4$ and $Ch(f) = d(f) - 4$. We define a few rules to redistribute the charges so that every vertex and face has a final charge of at least $0$.

(R1) Each $3$-vertex receives $\frac{1}{2}$ from each $5^+$-face it belongs to.

(R2) Each $3$-face receives $\frac{1}{2}$ from each adjacent $5^+$-face.

By Lemma~\ref{mindegree-2}, we only need to consider $3$-vertices and $4$-vertices. Since each $4$-vertex does not give or receive any charges, its charge remains $0$ until the end. By Lemma~\ref{3cycle4vertices-2}, a $3$-vertex cannot belong to any $3$-faces. By Lemma~\ref{3vertex4cycle-2}, a $3$-vertex can belong to at most one $4$-face. Therefore, it belongs to at least two $5^+$-faces and its final charge is at least $3-4+ 2 \cdot \frac{1}{2} = 0$. Every vertex has a final charge of at least $0$.

It remains to show every face has a final charge of at least $0$. A $4$-face has final charge $0$, since it does not give or receive any charges. By Lemma~\ref{no333-2},~\ref{no334-2}, and~\ref{no434-2}, a $3$-face is adjacent to at least two $5^+$-faces and its final charge is at least $3-4+ 2 \cdot \frac{1}{2} = 0$. Let $f$ be a $5$-face. By Lemma~\ref{no33-2}, $f$ can have at most two $3$-vertices. If $f$ has two $3$-vertices, then by Lemma~\ref{3cycle4vertices-2} and~\ref{no5cycled3d33cycle-2}, it cannot be adjacent to any $3$-faces, and its final charge is at least $5 - 4 - 2 \cdot \frac{1}{2} = 0$. If $f$ has exactly one $3$-vertex, then by Lemma~\ref{3cycle4vertices-2},~\ref{bowtie-2}, and~\ref{no5cycle3d3cycle3cycle-2}, $f$ can be adjacent to at most one $3$-face, and its final charge is at least $5 - 4 - 2 \cdot \frac{1}{2} = 0$. Otherwise, $f$ has no $3$-vertices. By Lemma~\ref{bowtie-2}, $f$ can be adjacent to at most two $3$-faces, and its final charge is at least $5 - 4 - 2 \cdot \frac{1}{2} = 0$. 

For an $\ell$-face $f$, where $\ell \ge 6$, let $a$ be the number of $3$-vertices in $V(f)$ and $b$ be the number of $3$-faces adjacent to $f$.  By Lemma~\ref{3cycle4vertices-2}, $2a + b \le \ell$. By Lemma~\ref{bowtie-2}, $2b+a \le \ell$. Therefore, $\frac{a+b}{2} \le \frac{\ell}{3}$. We conclude that $f$ has a final charge of at least
$$ \ell - 4 - \frac{1}{2} a - \frac{1}{2} b \ge \ell - 4 - \frac{\ell}{3} = \frac{2 \ell}{3} - 4  \ge 0,$$
since $\ell \ge 6$. This is a contradiction with~\eqref{discharging2}, and the proof is completed.

\section{Open Questions}
We believe studying packing $(1^j,2^k)$-colorings can provide insights on how to solve hard questions in square coloring, such as Wegner's Conjecture. We end this paper by proposing two open questions.

\begin{op}
What is the minimum positive integer $k_1$ such that every planar graph with maximum degree at most four is packing $(1,2^{k_1})$-colorable?    
\end{op}

\begin{op}
What is the minimum positive integer $k_2$ such that every planar graph with maximum degree at most four is packing $(1^2,2^{k_2})$-colorable?    
\end{op}

We have shown in this paper that $6 \le k_1 \le 10$ and $4 \le k_2 \le 7$. We feel the upper bounds may be the best achievable bounds with the current method. Any further improvement appears to require a new idea.

\section*{Acknowledgement}

The authors would like to thank Bernard Lidick\'{y} and Yan Wang for the helpful discussions. The C++ program is adapted from their earlier code.

\section{Appendix}
We explain in detail how our $C++$ programs work. We explain only for Lemma~\ref{no434-2} since the core algorithms are the same. We notice that there are two types of structures in the boundary vertices, i.e., a $K_{1,2}$ and a $K_{1,3}$. We call the first type {\em small claw} and the second type {\em big claw}. We first analyze the number of possible colorings for each small claw and each big claw up to symmetry. 

For small claws, since there are two $1$-colors ($1_a, 1_b$) and seven $2$-colors ($2_1, 2_2, \ldots, 2_7$), there are $84$ different colorings. Specifically, we classify these colorings into four types (see Figure~\ref{small_claws}). 
Recall that we use $-1$ as an undetermined $1$-color. If all leaf vertices are assigned $2$-colors, we use $-1$ at the central vertex, indicating it can take either $1_a$ or $1_b$.
For Type 1, we choose two $2$-colors out of seven, yielding $\binom{7}{2} = 21$ colorings. In Type 2, the $1$-colors $1_a$ and $1_b$ can be switched, and we choose one $2$-color, giving $2 \times \binom{7}{1} = 14$ colorings. For Type 3, the exact choice between $1_a$ and $1_b$ does not affect the extendability of the coloring, because the vertex assigned the $1$-color is not the central vertex and is at a distance of at least two from all uncolored vertices; however, there are $7\times6$ ways to select the $2$-colors, resulting $42$ cases. Finally, Type 4 involves choosing a single $2$-color, providing exactly $\binom{7}{1}=7$ choices. 

To optimize our algorithm, we do not explicitly check Type 1', Type 2', and Type 4' in our program, as they reduce to Type 1, Type 2, and Type 4, respectively. Suppose a coloring of boundary vertices assigns a Type 1' coloring to a specific small claw. By recoloring the central vertex of that claw with $-1$, we obtain a coloring of boundary vertices where that claw is of Type 1. Since our program exhaustively verifies that all valid combinations containing Type 1 are extendable, the extendability of the original combination involving Type 1' is guaranteed. 
Since if a $1$-color is used in the leaf vertices of a small claw, then it will not influence the extendability, Type 2' is reduced to Type 2 and Type 4' is reduced to Type 4. Note that Type 3 cannot be reduced to Type 2, since there is often an edge joining two small claws in $G$.

\begin{figure}[ht]
\vspace{-23mm}
 \hspace{-18mm} 
  \includegraphics[scale=0.95]{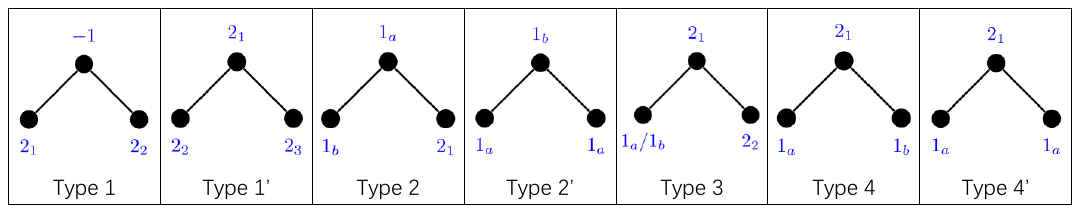}
  \vspace{-233mm}
\caption{Colorings of Small Claws}\label{small_claws}
\vspace{-1mm}
\end{figure}

\begin{figure}[ht]
\vspace{-22mm}
\hspace{2mm} 
  \includegraphics[scale=0.75]{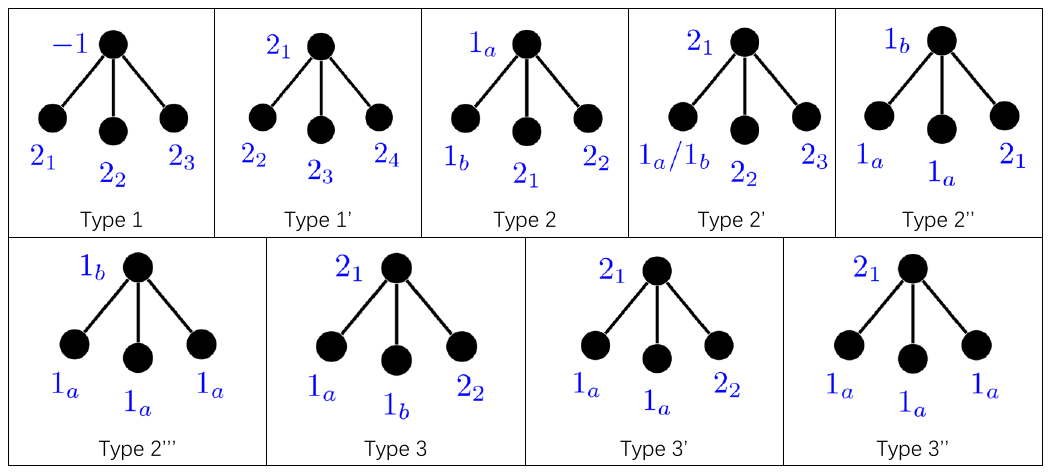}
  \vspace{-145mm}
\caption{Colorings of Big Claws}\label{big_claws}
\vspace{-3mm}
\end{figure}


For big claws, by similar analysis, there are $119$ different colorings. Specifically, we classify them into three types (see Figure~\ref{big_claws}). For Type 1, we choose three $2$-colors out of seven, resulting in $\binom{7}{3}=35$ colorings. In Type 2, $1_a$ and $1_b$ can be switched, and we choose two $2$-colors, giving $2\times\binom{7}{2}=42$ colorings. Finally, for Type 3, the exact choice between $1_a$ and $1_b$ does not affect the extendability of the coloring, because the vertex assigned the $1$-color is not the central vertex and is at a distance of at least two from all uncolored vertices; however, there are $7\times6$ ways to select the $2$-colors, resulting in $42$ cases.

To optimize our algorithm, we do not explicitly check Types 1', 2', 2'', 2''', 3', and 3'', as they can reduce to Types 1, 2 or 3. Using the same reduction method as for small claws, the extendability of configurations containing these colorings is guaranteed through the following steps: recoloring the central vertex with $-1$ reduces Type 1' to Type 1; recoloring the central vertex with an available $1$-color reduces Type 2' to Type 2; recoloring one or two leaves with available $2$-colors reduces Type 2'' and Type 2''' to Type 2; recoloring a leaf with the alternative $1$-color reduces Type 3' to Type 3; and recoloring one leaf with the alternative $1$-color and another with an available $2$-color reduces Type 3'' to Type 3.

The full code are available at ``https://sites.google.com/view/xujunliu1993/code''. The input file gives the number of small claws, the vertex number of each small claws, the number of big claws, the vertex number of each big claws, and the incidence matrix of the graph. Color $0$ means the vertex is uncolored.


\begin{figure}[ht]
\vspace{-23mm}
 \hspace{-18mm} \includegraphics[scale=0.95]{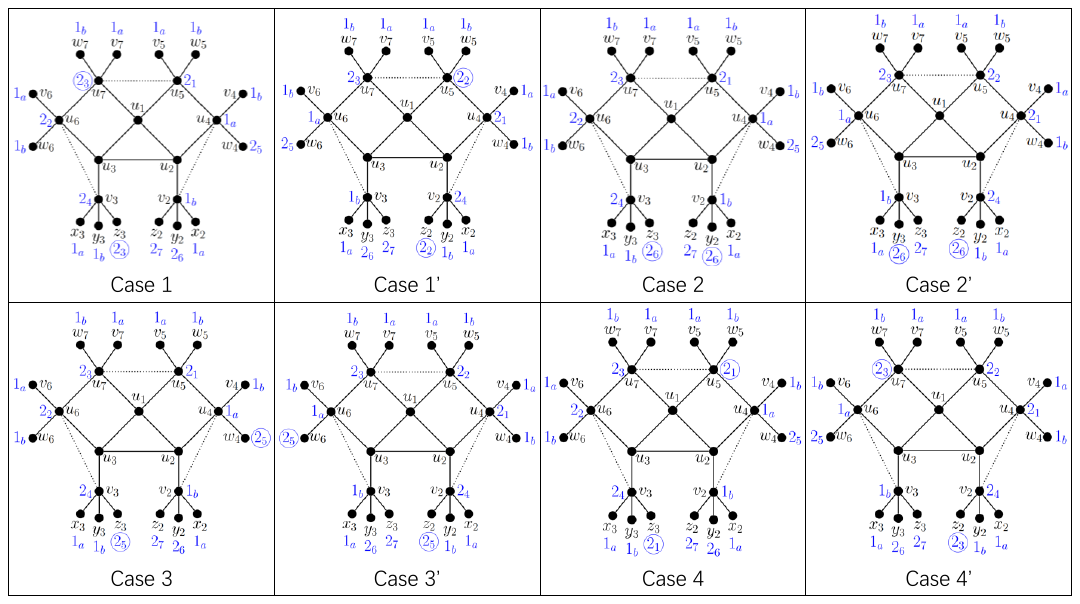}
  \vspace{-168mm}
\caption{Badcases for Lemma~\ref{no434-2}.}\label{434-badcases}
\end{figure}

\begin{figure}[H]
\vspace{-30mm}
\hspace{-17mm} \includegraphics[scale=0.95]{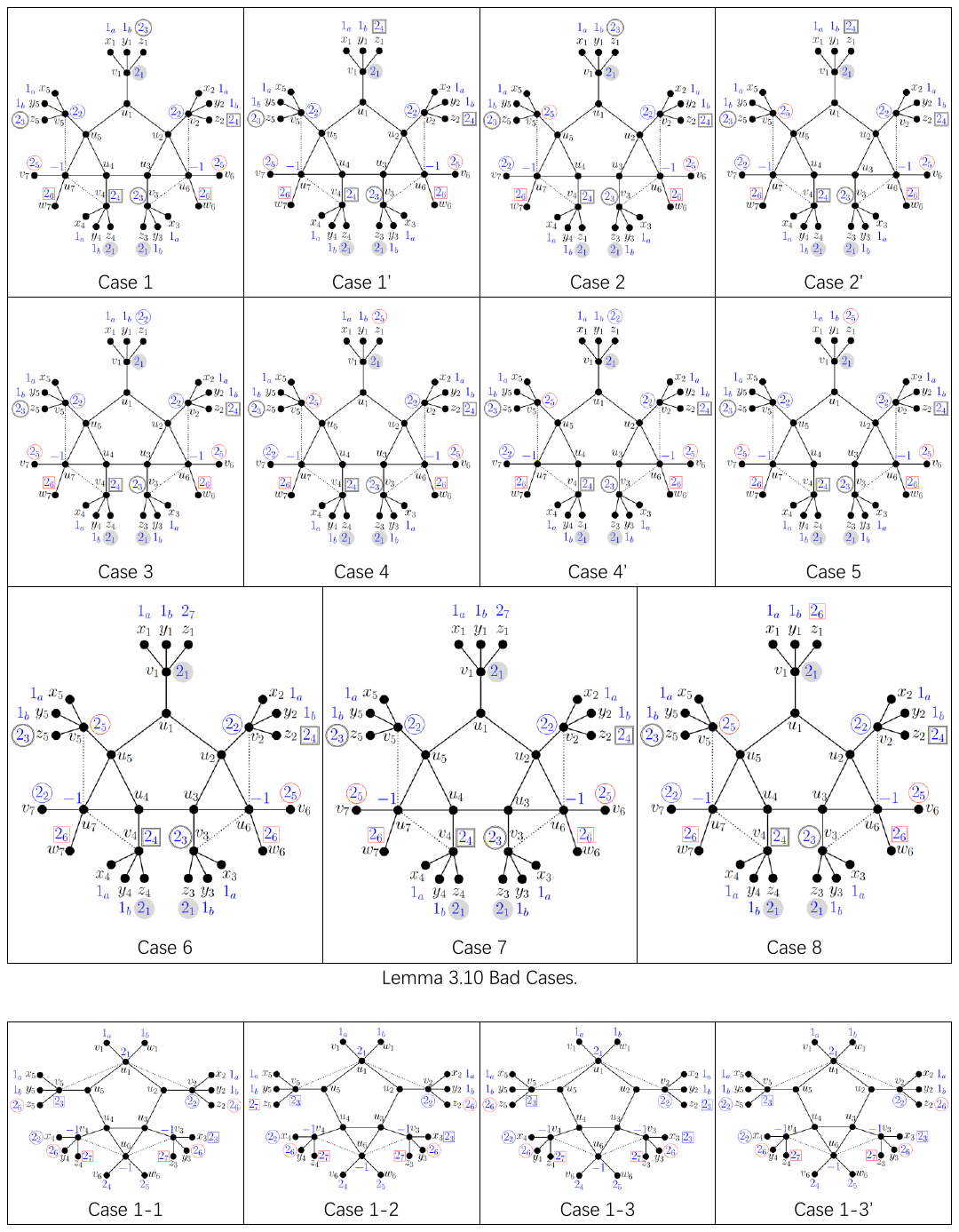}
\vspace{-37mm}
\caption{Bad cases for Lemma~\ref{no5cycle3d3cycle3cycle-2} and Lemma~\ref{no5cycled3d33cycle-2}.}\label{3d533-badcases}
\vspace{-2mm}
\end{figure}

\begin{figure}[H]
\vspace{-30mm}
 \hspace{-18mm} \includegraphics[scale=0.95]{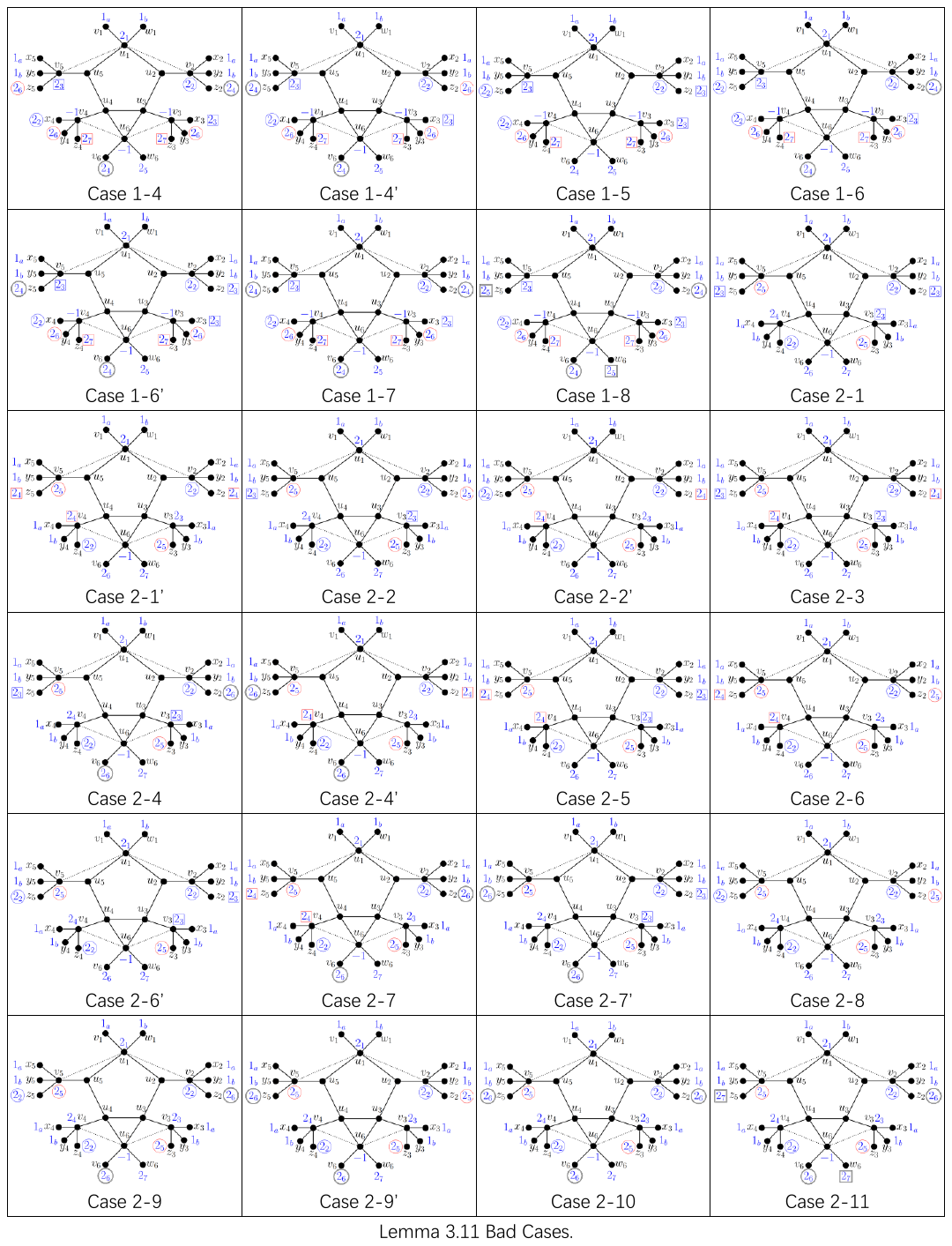}
 \vspace{-38mm}
\caption{Bad cases for Lemma~\ref{no5cycled3d33cycle-2}.}\label{3d3d53-badcases}
\vspace{-2mm}
\end{figure}


\subsection{C++ Program}

\begin{figure}[H]
 \hspace{-3mm} \includegraphics[scale=0.44]{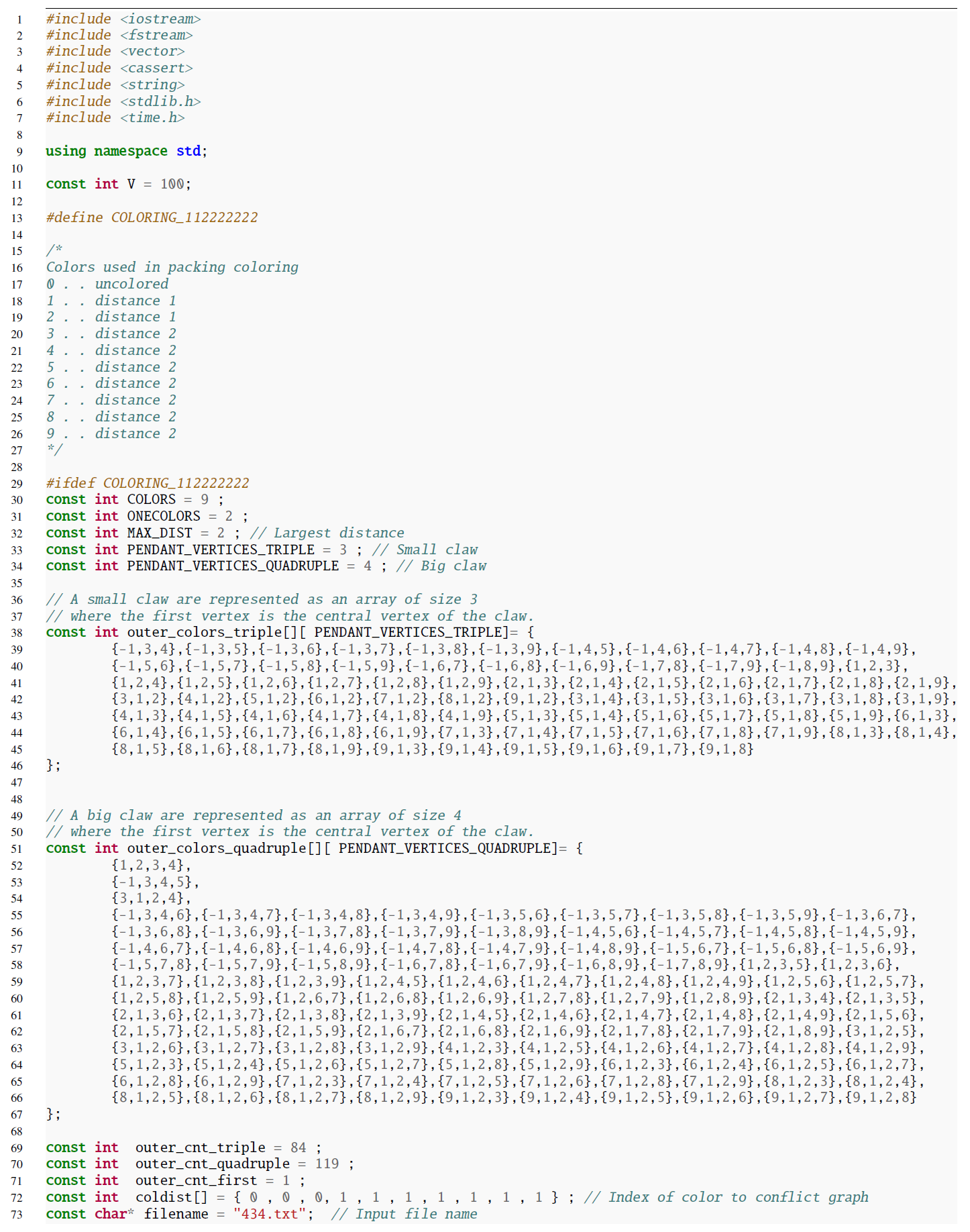}
\end{figure}

\begin{figure}[H]
 \hspace{-3mm} \includegraphics[scale=0.46]{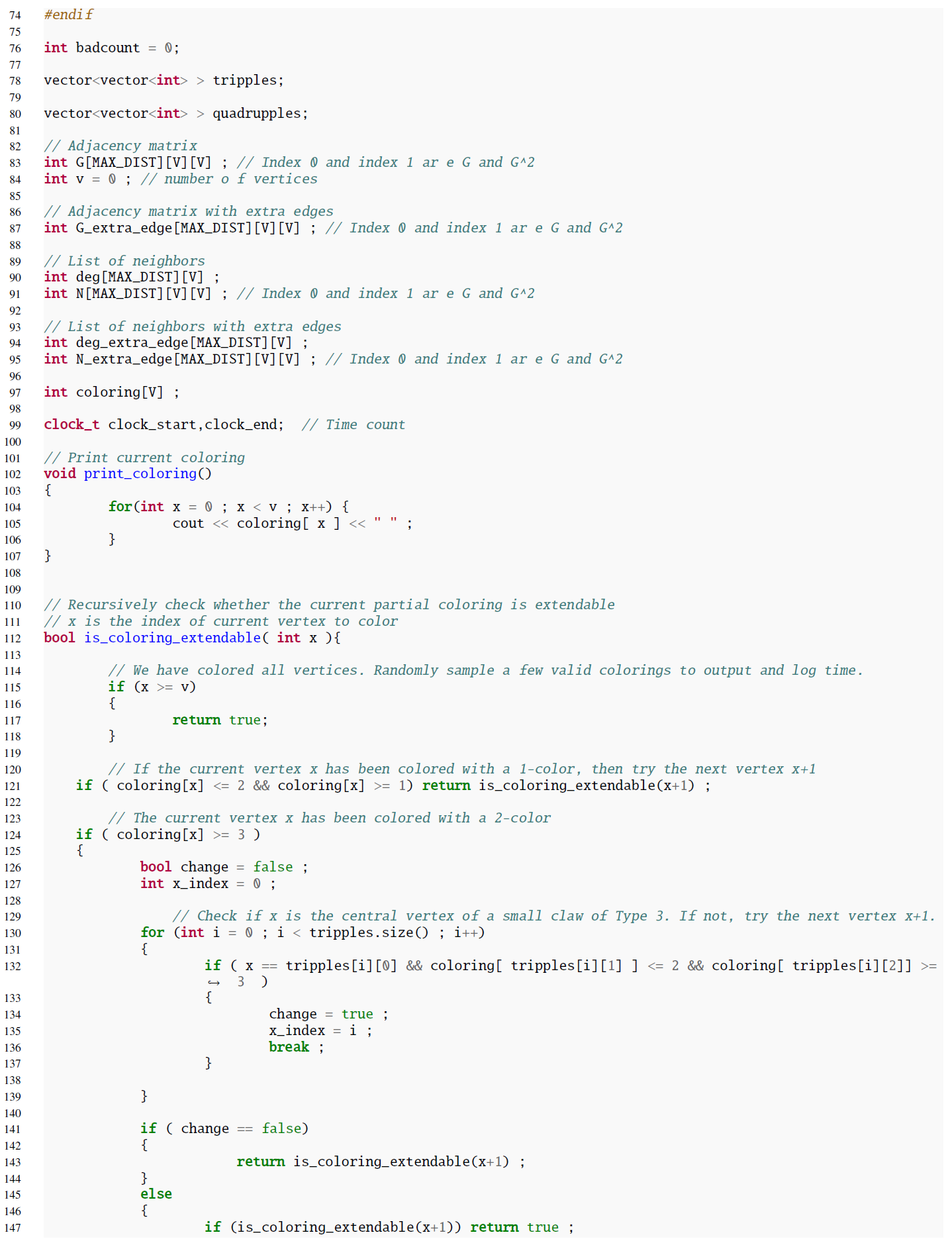}
\end{figure}

\begin{figure}[H]
 \hspace{-3mm} \includegraphics[scale=0.46]{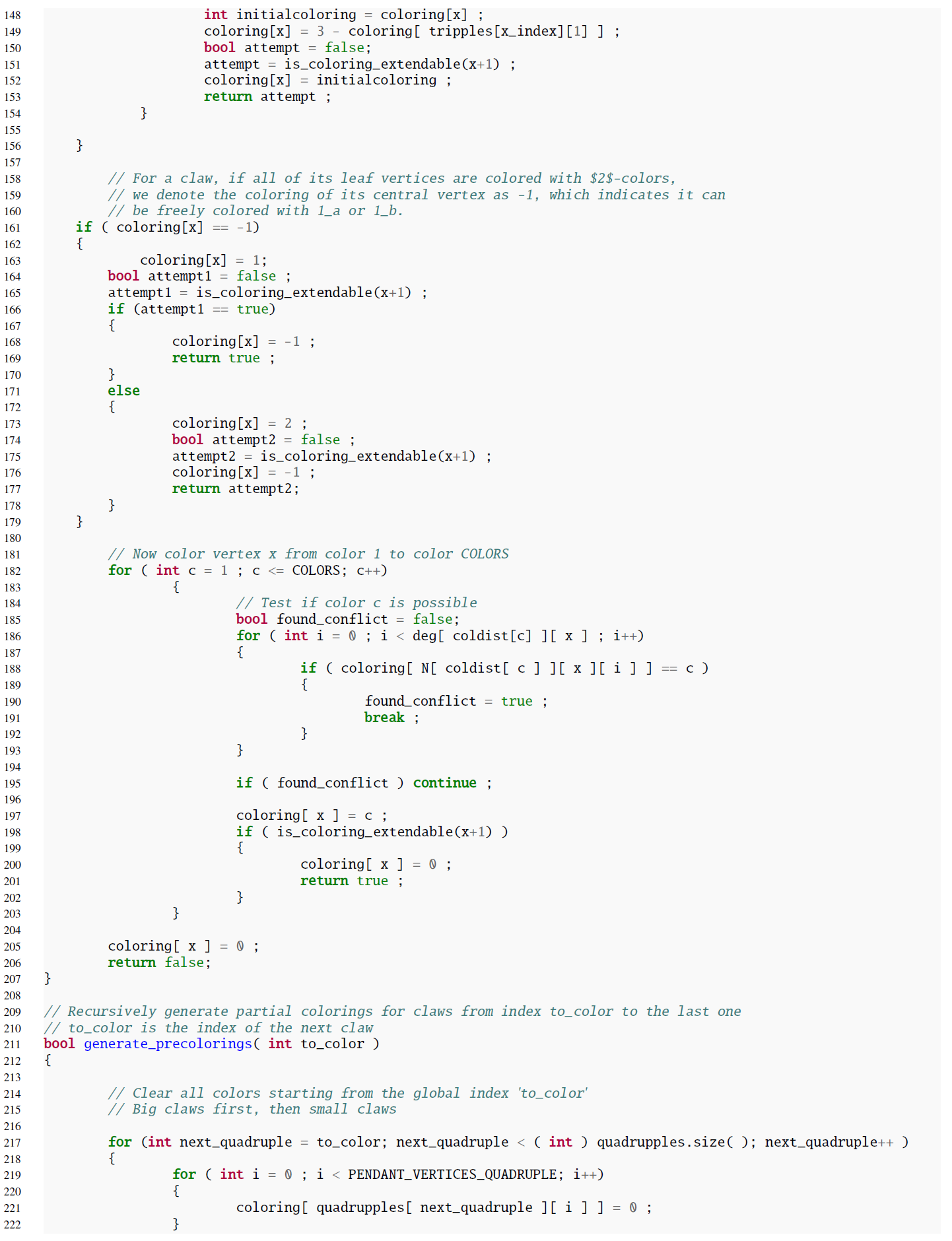}
\end{figure}

\begin{figure}[H]
 \hspace{-3mm} \includegraphics[scale=0.46]{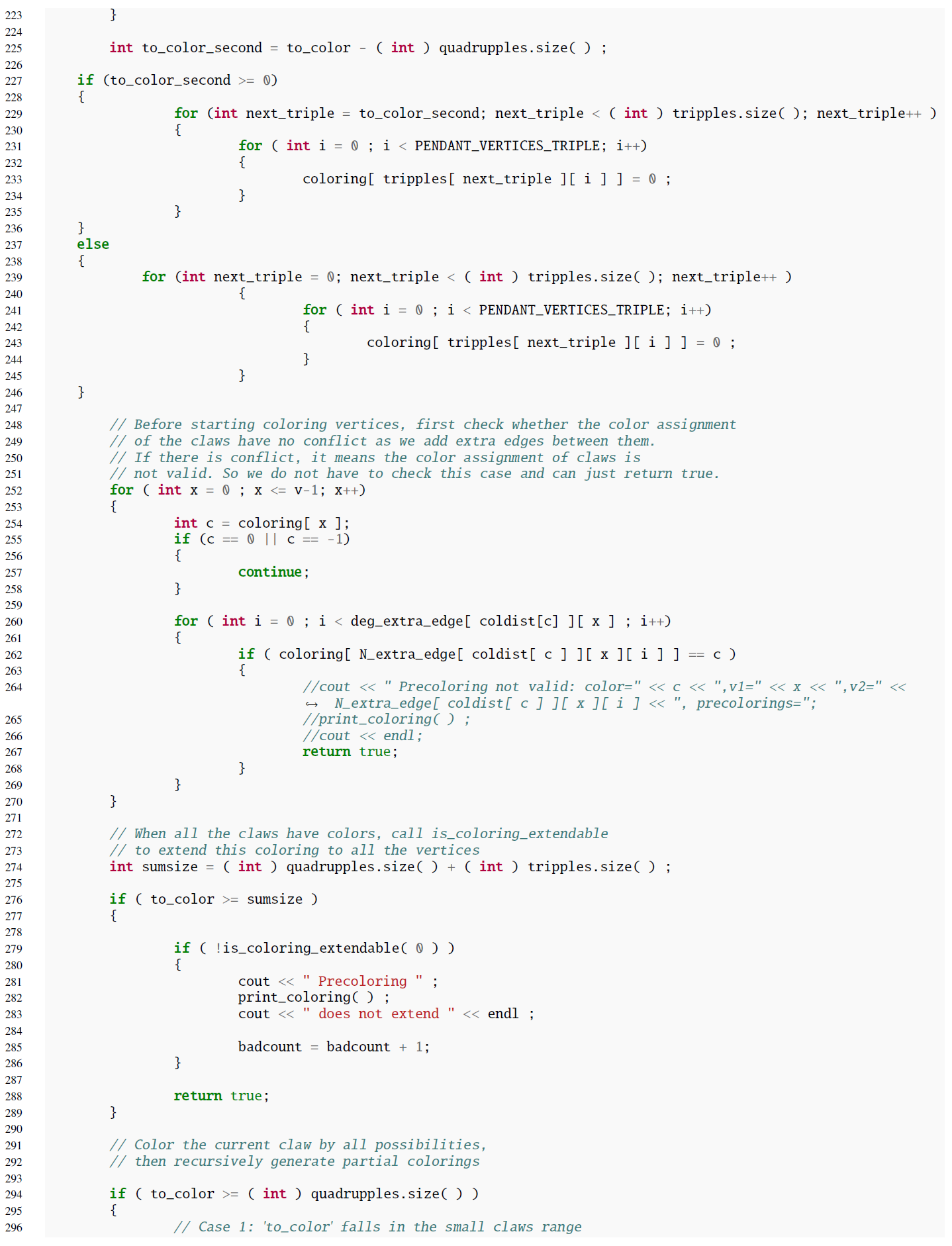}
\end{figure}

\begin{figure}[H]
 \hspace{-3mm} \includegraphics[scale=0.46]{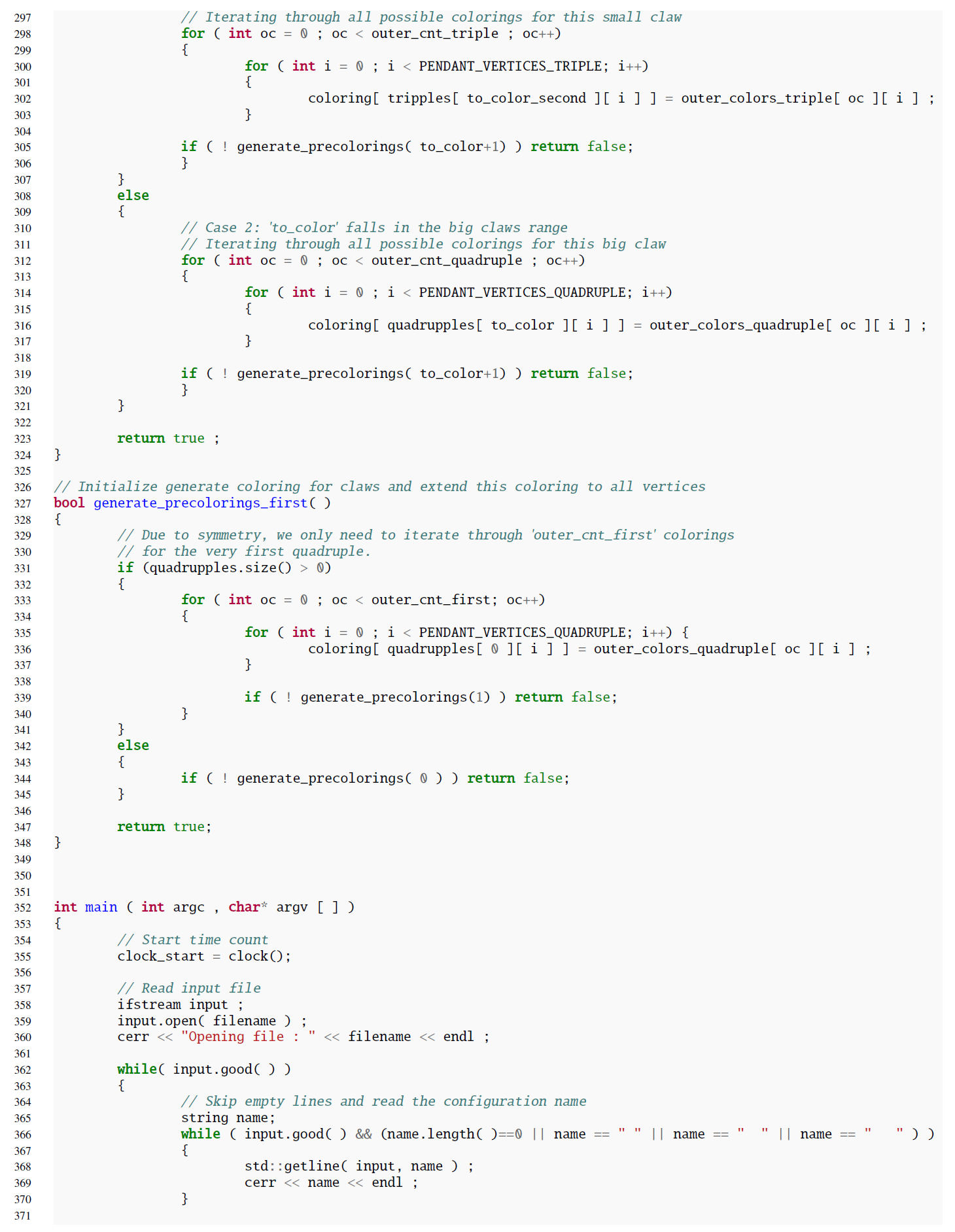}
\end{figure}

\begin{figure}[H]
 \hspace{-3mm} \includegraphics[scale=0.46]{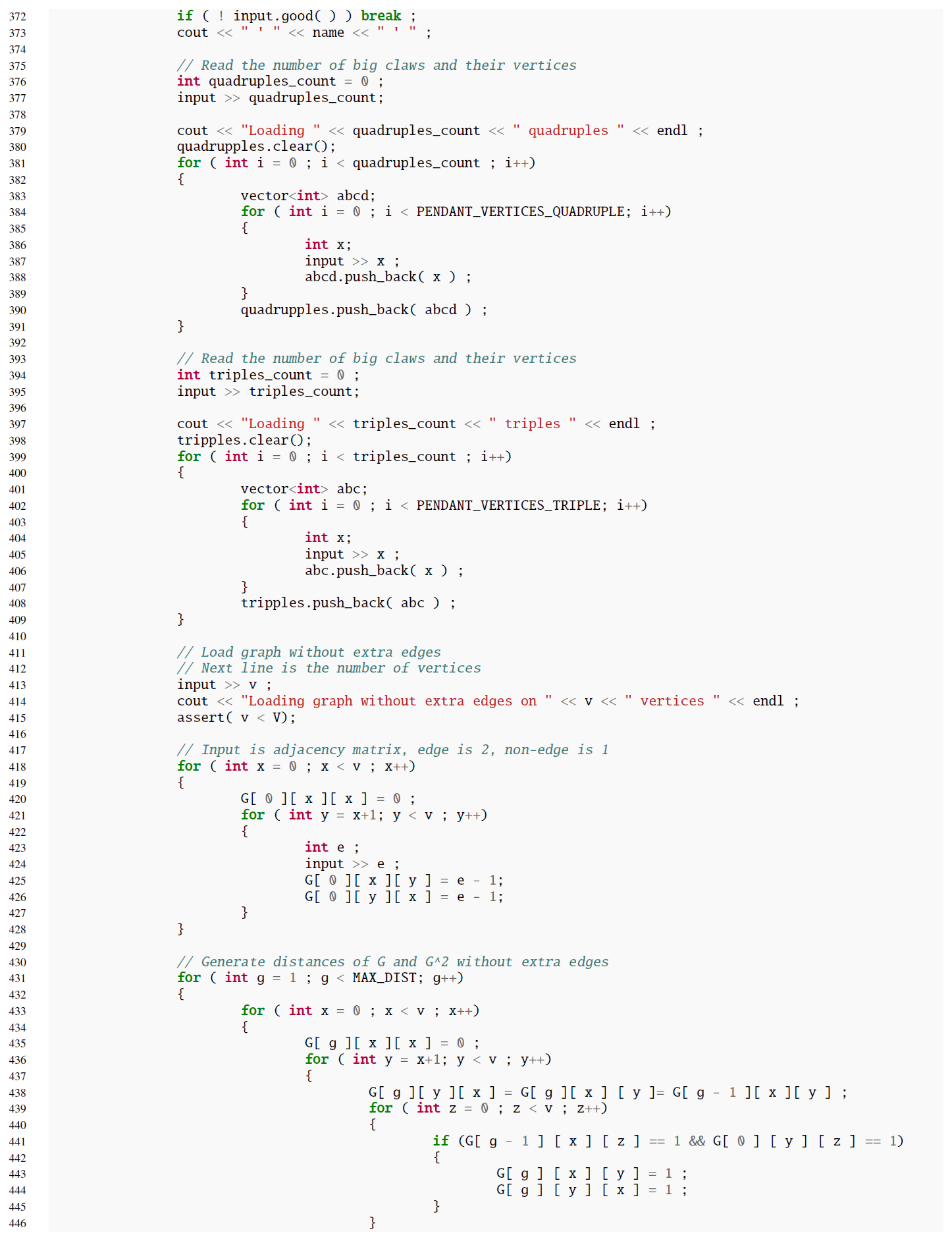}
\end{figure}

\begin{figure}[H]
 \hspace{-3mm} \includegraphics[scale=0.46]{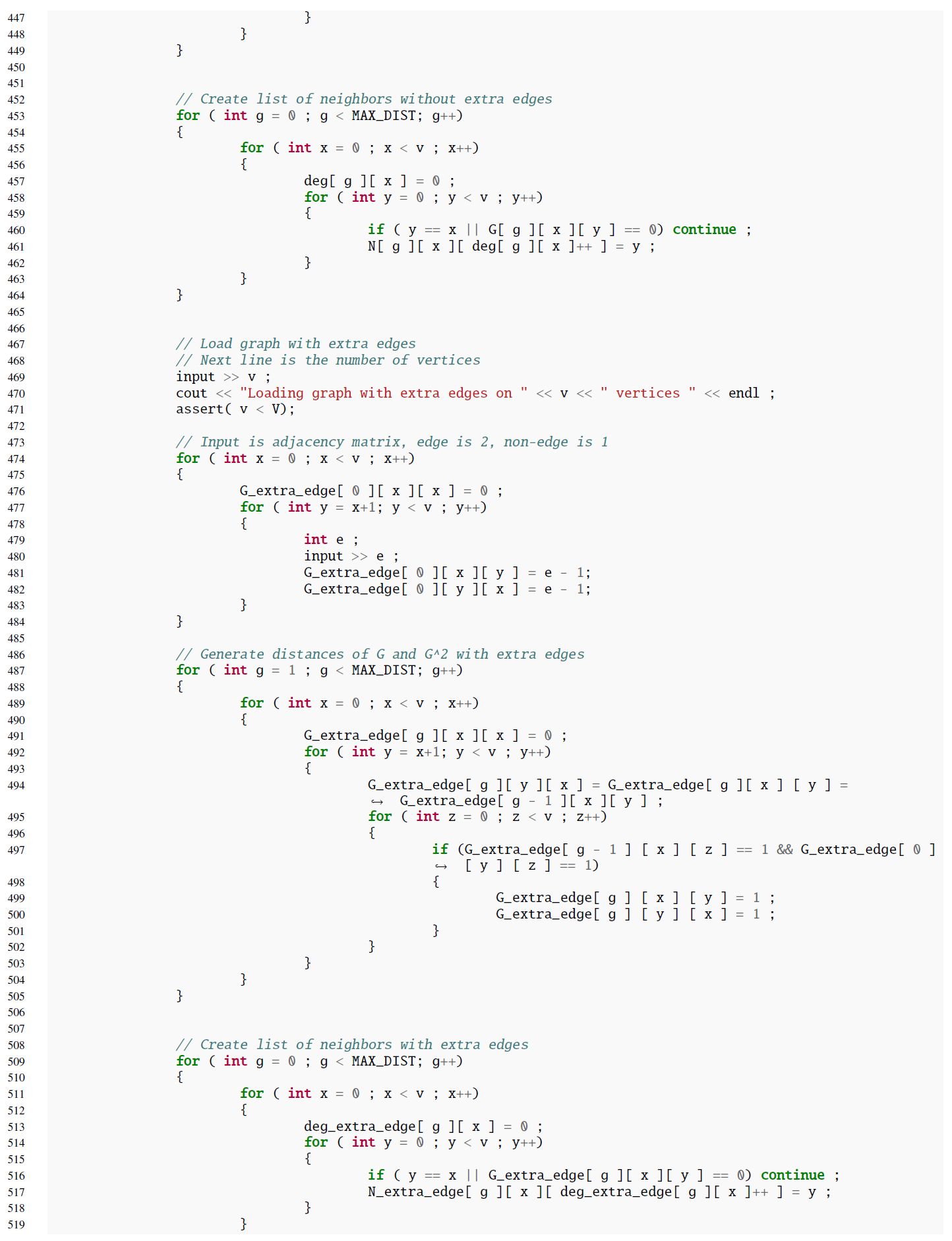}
\end{figure}

\begin{figure}[H]
 \hspace{-3mm} \includegraphics[scale=0.46]{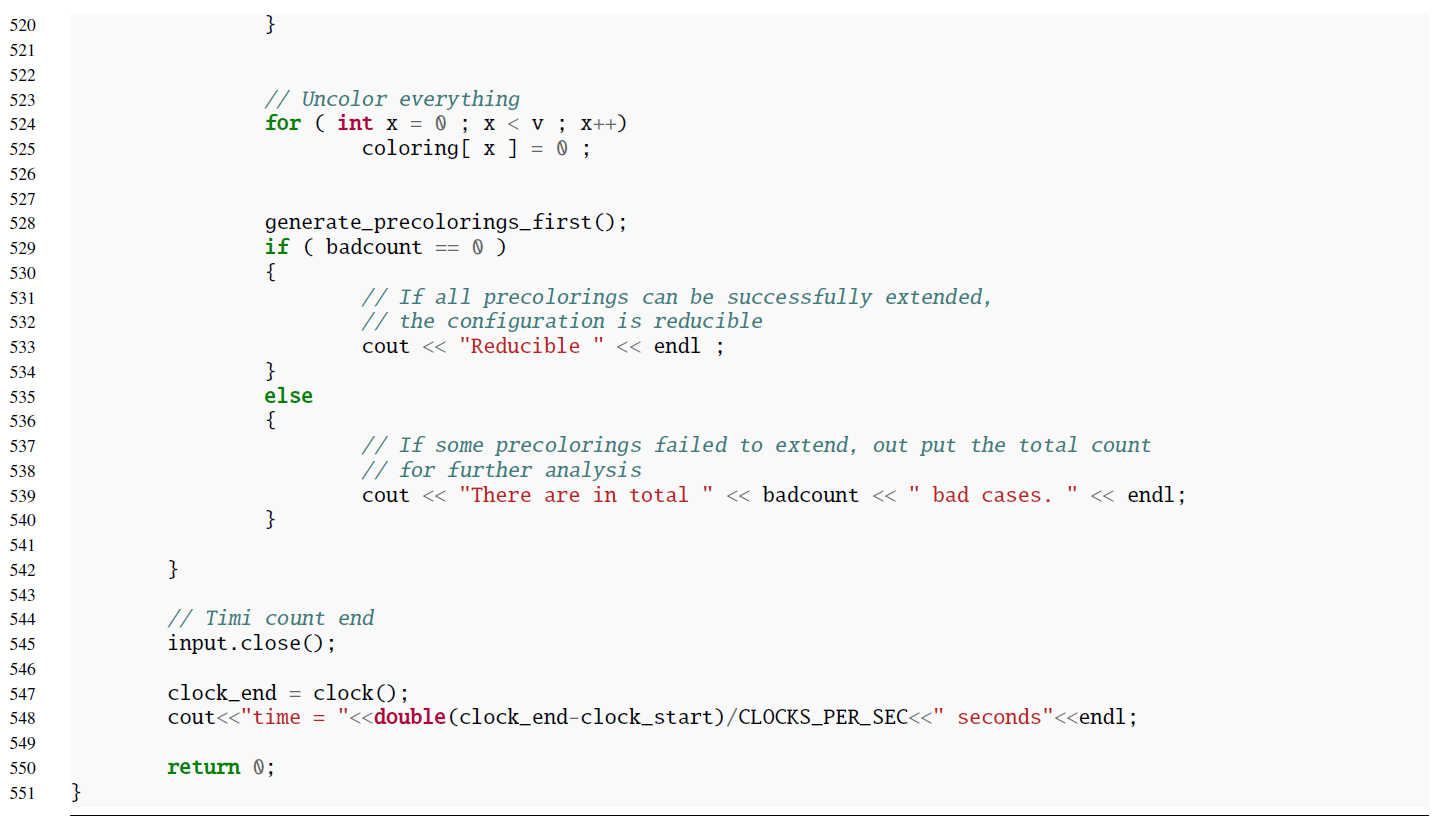}
\end{figure}

\subsection{Sample Input File}

As an example, we detail the structure of one of the input files used for Lemma~\ref{no434-2}. Here the first line is the name of our target configuration. The second and fifth lines indicate the number of big claws and small claws, respectively. Lines 3 and 4 list the vertex indices for the big claws, with the first integer in each line denoting the central vertex. Similarly, lines 6 through 9 list the vertex indices for the small claws, again starting with the central vertex. The tenth line provides the total number of vertices followed by the adjacency matrix of the original configuration. Finally, the eleventh line gives the number of vertices and the adjacency matrix of the configuration after the extra edges are added. Note that we use $2$ to denote edge and $1$ to denote non-edge. Since the adjacency matrices are symmetric and all diagonal entries are $1$, we only input their strictly upper triangular parts.

\begin{figure}[H]
 \hspace{-3mm} \includegraphics[scale=0.46]{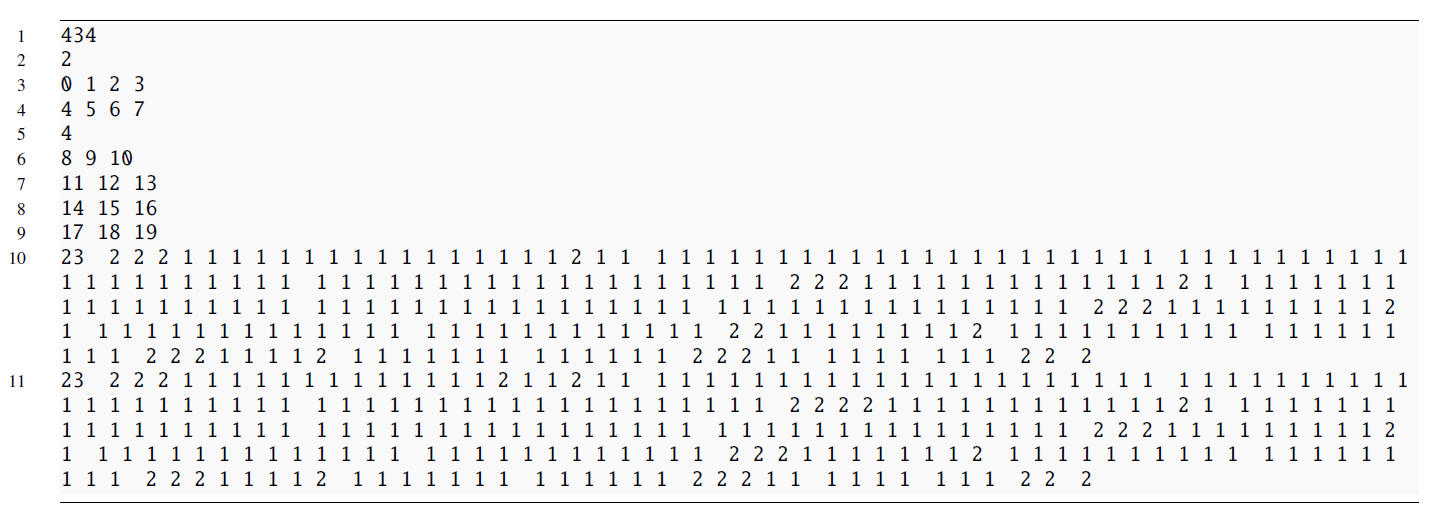}
\end{figure}

\end{document}